\documentclass{birkmult}

\usepackage{amssymb}
\usepackage[mathscr]{eucal}

\newtheorem{theorem}{Theorem}[section]
\newtheorem{lemma}[theorem]{Lemma}

\theoremstyle{definition}
\newtheorem{definition}[theorem]{Definition}

\numberwithin{equation}{section}
\parskip=\smallskipamount

\newcommand{\bbC}{{\mathbb C}}
\renewcommand{\lg}{\eL(\eG)}

\newcommand{\lhk}{\eL(\eH,\eK)}

\newcommand{\lh}{\eL(\eH)}
\newcommand{\lk}{\eL(\eK)}
\newcommand{\lkperp}{\eL(\eK^\perp)}

\newcommand{\diag}{{\rm diag\,}}

\newcommand{\ran}{{\rm ran\,}}

\newcommand{\ip}[2]{\left<#1,#2\right>}

\newcommand{\eC}{{\mathfrak C}}

\newcommand{\eF}{{\mathfrak F}}
\newcommand{\eG}{{\mathfrak G}}
\newcommand{\eH}{{\mathfrak H}}
\newcommand{\eK}{{\mathfrak K}}
\newcommand{\eL}{{\mathfrak L}}
\newcommand{\eM}{{\mathfrak M}}

\newcommand{\eP}{{\mathfrak P}}

\newcommand{\fr}{Fej\'er-Riesz }

\newcommand{\bbD}{{\mathbb D}}
\newcommand{\bbT}{{\mathbb T}}
\newcommand{\thecircle}{{\bbT}}
\newcommand{\ranclosure}{\overline{{\rm ran\vphantom{+}}}\,}

\begin{document}

\title[The Operator Fej\'er-Riesz Theorem]
{The Operator Fej\'er-Riesz Theorem}

\author{Michael~A.~Dritschel} 

\address{School of Mathematics and
  Statistics, Herschel Building, University of Newcastle,
  Newcastle upon Tyne NE1 7RU, UK}

\email{m.a.dritschel@ncl.ac.uk}

\author{James Rovnyak} 

\address{University of Virginia, Department of Mathematics, P.~O.\ Box
  400137, Charlottes\-ville, VA 22904--4137}
\email{rovnyak@virginia.edu} 

\subjclass{Primary 47A68; Secondary 60G25, 47A56, 47B35, 42A05, 32A70,
  30E99}

\keywords{Trigonometric polynomial, Fejer-Riesz theorem, spectral
factorization, Schur complement, noncommutative polynomial, Toeplitz
operator, shift operator}

\dedicatory{To the memory of Paul Richard Halmos.}

\begin{abstract}
  
  The \fr theorem has inspired numerous generalizations in one and
  several variables, and for matrix- and operator-valued functions.
  This paper is a survey of some old and recent topics that center
  around Rosenblum's operator generalization of the classical \fr
  theorem.

\end{abstract}

\maketitle

\tableofcontents

\section{Introduction}\label{S:intro}
\bigskip

The classical \fr factorization theorem gives the form of a
nonnegative trigonometric polynomial on the real line, or,
equivalently, a Laurent polynomial that is nonnegative on the unit
circle.  For the statement, we write $\bbD = \{ z \colon |z|<1\}$ and
$\thecircle = \{ \zeta \colon |\zeta|=1\}$ for the open unit disk and unit
circle in the complex plane.

\medskip
\noindent {\bf \fr Theorem.}  {\it
  A Laurent polynomial \smash{$q(z) = \sum_{k=-m}^m q_k z^k$} which
  has complex coefficients and satisfies $q(\zeta) \ge 0$ for all
  $\zeta\in\thecircle$ can be written
\begin{equation*}
q(\zeta) = |p(\zeta)|^2 , \qquad \zeta \in\thecircle ,
\end{equation*}
for some polynomial $p(z) = p_0 + p_1 z + \cdots + p_m z^m$, and
$p(z)$ can be chosen to have no zeros in $\bbD$.
}
\medskip

The original sources are Fej\'er \cite{Fejer1916} and Riesz
\cite{Riesz1916}.  The proof is elementary and consists in showing
that the roots of $q(z)$ occur in pairs $z_j$ and $1/\bar z_j$ with
$|z_j| \ge 1$.  Then the required polynomial $p(z)$ is the product of
the factors $z - z_j$ adjusted by a suitable multiplicative
constant~$c$.  Details appear in many places; see e.g.\ \cite[p.\ 
20]{GSoriginal}, \cite[p.\ 235]{HeltonPutinar}, or \cite[p.\ 
26]{simon1}.

The \fr theorem arises naturally in spectral theory, the theory of
orthogonal polynomials, prediction theory, moment problems, and
systems and control theory.  Applications often require
generalizations to functions more general than Laurent polynomials,
and, more than that, to functions whose values are matrices or
operators on a Hilbert space.  The spectral factorization problem is
to write a given nonnegative matrix- or operator-valued function $F$
on the unit circle in the form $F=G^*G$ where $G$ has an analytic
extension to the unit disk (in a suitably interpreted sense).  The
focal point of our survey is the special case of a Laurent polynomial
with operator coefficients.

The operator \fr theorem (Theorem \ref{ofr}) obtains a conclusion
similar to the classical result for Laurent polynomial whose
coefficients are Hilbert space operators: if $Q_j$, $j=-m,\dots,m$,
are Hilbert space operators such that
\begin{equation}\label{jan19a}
  Q(\zeta) = \sum_{k=-m}^m Q_k \zeta^k \ge 0, \qquad \zeta \in \thecircle,
\end{equation}
then there is a polynomial $P(z) = P_0 + P_1 z + \cdots + P_m z^m$
with operator coefficients such that
\begin{equation}\label{jan19b}
  Q(\zeta) = P(\zeta)^*P(\zeta), \qquad \zeta \in \thecircle.
\end{equation}
This was first proved in full generality in 1968 by Marvin
Rosenblum~\cite{Rosenblum1968}.  The proof uses Toeplitz operators and
a method of Lowdenslager, and it is a fine example of operator theory
in the spirit of Paul Halmos.  Rosenblum's proof is reproduced
in~\S\ref{S:FejerRiesz}.

Part of the fascination of the operator \fr theorem is that it can be
stated in a purely algebraic way.  The hypothesis \eqref{jan19a} on
$Q(z)$ is equivalent to the statement that an associated Toeplitz
matrix is nonnegative.  The conclusion \eqref{jan19b} is equivalent to
$2m+1$ nonlinear equations whose unknowns are the coefficients
$P_0,P_1,\dots,P_m$ of $P(z)$.  Can it be that this system of
equations can be solved by an algebraic procedure?  The answer is,
yes, and this is a recent development.  The iterative procedure uses
the notion of a Schur complement and is outlined in \S\ref{S:Schur}.

There is a surprising connection between Rosenblum's proof of the
operator \fr theorem and spectral factorization.  The problem of
spectral factorization is formulated precisely in \S\ref{S:spectral},
using Hardy class notions.  A scalar prototype is Szeg{\H o}'s theorem
(Theorem \ref{nov23d}) on the representation of a positive integrable
and log-integrable function $w$ on the unit circle in the form $|h|^2$
for some $H^2$ function~$h$.  The operator and matrix counterparts of
Szeg{\H o}'s theorem, Theorems \ref{nov23c} and \ref{nov23cc}, have
been known for many years and go back to fundamental work in the 1940s
and 1950s which was motivated by applications in prediction theory
(see the historical notes at the end of \S\ref{S:spectral}).  We
present a proof that is new to the authors and we suspect not widely
known.  It is based on Theorem \ref{nov20c}, which traces its origins
to Rosenblum's implementation of the Lowdenslager method.  In
\S\ref{S:spectral} we also state without proof some special results
that hold in the matrix case.

The method of Schur complements points the way to an approach to
multivariable factorization problems, which is the subject
of~\S\ref{S:multivar}.  Even in the scalar case, the obvious first
ideas for multivariable generalizations of the \fr theorem are false
by well-known examples.  Part of the problem has to do with what one
might think are natural restrictions on degrees.  In fact, the
restrictions on degrees are not so natural after all.  When they are
removed, we can prove a result, Theorem \ref{jan16a}, that can be
viewed as a generalization of the operator \fr theorem in the strictly
positive case.  We also look at the problem of outer factorization, at
least in some restricted settings.

In recent years there has been increasing interest in noncommutative
function theory, especially in the context of functions of freely
noncommuting variables.  In \S{6} we consider noncommutative analogues
of the $d$-torus, and corresponding notions of nonnegative
trigonometric polynomials.  In the freely noncommutative setting,
there is a very nice version of the \fr theorem
(Theorem~\ref{jan24d}).  In a somewhat more general noncommutative
setting, which also happens to cover the commutative case as well, we
have a version of Theorem~\ref{jan16a} for strictly positive
polynomials (Theorem~\ref{jan12a}).

Our survey does not aim for completeness in any area.  In particular,
our bibliography represents only a selection from the literature.  The
authors regret and apologize for omissions.

\goodbreak
\section{The operator \fr theorem}\label{S:FejerRiesz}

In this section we give the proof of the operator \fr theorem by
Rosenblum \cite{Rosenblum1968}.  The general theorem had precursors.
A finite-dimensional version was given by Rosenblatt
\cite{Rosenblatt1958}, an infinite-dimensional special case by Gohberg
\cite{Gohberg1964}.

We follow standard conventions for Hilbert spaces and operators.  If
$A$ is an operator, $A^*$ is its adjoint.  Norms of vectors and
operators are written $\| \cdot \|$.  Except where noted, no
assumption is made on the dimension of a Hilbert space, and
nonseparable Hilbert spaces are allowed.

\begin{theorem}[Operator Fej\'er-Riesz Theorem]\label{ofr}
  Let $Q(z) = \sum_{k=-m}^m Q_k z^k$ be a Laurent polynomial with
  coefficients in $\lg$ for some Hilbert space~$\eG$.  If $Q(\zeta) \ge 0$
  for all $\zeta\in\thecircle$, then
\begin{equation}\label{aug27c}
Q(\zeta) = P(\zeta)^*P(\zeta) , \qquad \zeta \in\thecircle,
\end{equation}
for some polynomial $P(z) = P_0 + P_1 z + \cdots + P_m z^m$ with
coefficients in $\lg$.  The polynomial $P(z)$ can be chosen to be
outer.
\end{theorem}

The definition of an outer polynomial will be given later; in the
scalar case, a polynomial is outer if and only if it has no zeros in
$\bbD$.

The proof uses (unilateral) shift and Toeplitz operators (see
\cite{brownhalmos} and \cite{halmos1961}).  By a {\bf shift operator}
here we mean an isometry $S$ on a Hilbert space $\eH$ such that the
unitary component of $S$ in its Wold decomposition is trivial.  With
natural identifications, we can write $\eH = \eG \oplus \eG \oplus
\cdots$ for some Hilbert space $\eG$ and
\begin{equation*}
S (h_0,h_1,\dots) = (0,h_0,h_1,\dots)  
\end{equation*}
when the elements of $\eH$ are written in sequence form.  Suppose
that such a shift $S$ is chosen and fixed.  If $T,A\in\lh$, we say
that $T$ is {\bf Toeplitz} if $S^*TS=T$, and that $A$ is {\bf
  analytic} if $AS=SA$.  An analytic operator $A$ is said to be {\bf
  outer} if $\ranclosure A$ is a subspace of $\eH$ of the form $\eF
\oplus \eF \oplus \cdots$ for some closed subspace $\eF$ of $\eG$.

As block operator matrices, Toeplitz and analytic operators have the
forms
  \begin{equation}\label{sep30a}
  T = \begin{pmatrix}
    T_0 & T_{-1} & T_{-2} & \cdots\\
    T_1 & T_0    & T_{-1} & \ddots \\
    T_2 & T_1  & T_0 & \ddots \\
    \vdots    & \ddots  & \ddots & \ddots
  \end{pmatrix} ,
\qquad
 A = \begin{pmatrix}
    A_0 & 0 & 0 & \cdots\\
    A_1 & A_0    & 0 & \ddots \\
    A_2 & A_1  & A_0 & \ddots \\
    \vdots    & \ddots  & \ddots & \ddots
  \end{pmatrix} .
\end{equation}
Here
\begin{equation}
  \label{aug28b}
  T_j = 
  \begin{cases}
    E_0^* {S^*}^j T E_0 \vert \eG ,     &\quad j \ge 0, \\[3pt]
    E_0^* T S^{|j|} E_0 \vert \eG,      &\quad j < 0,
  \end{cases}
\end{equation}
where $E_0 g = (g,0,0,\dots)$ is the natural embedding of $\eG$ into
$\eH$.  For examples, consider Laurent and analytic polynomials $Q(z)
= \sum_{k=-m}^m Q_k z^k$ and $P(z) = P_0 + P_1 z + \cdots + P_m z^m$
with coefficients in $\lg$.  Set $Q_j=0$ for $|j|>m$ and $P_j=0$ for
$j>m$.  Then the formulas
  \begin{equation}\label{aug31a}
  T_Q = \begin{pmatrix}
    Q_0 & Q_{-1} & Q_{-2} & \cdots\\
    Q_1 & Q_0    & Q_{-1} & \ddots \\
    Q_2 & Q_1  & Q_0 & \ddots \\
    \vdots    & \ddots  & \ddots & \ddots
  \end{pmatrix} ,
\qquad
 T_P = \begin{pmatrix}
    P_0 & 0 & 0 & \cdots\\
    P_1 & P_0    & 0 & \ddots \\
    P_2 & P_1  & P_0 & \ddots \\
    \vdots    & \ddots  & \ddots & \ddots
  \end{pmatrix} 
\end{equation}
define bounded operators on $\eH$.  Boundedness follows from the
identity
\begin{equation}\label{aug28a}
    \int_\thecircle {\ip{Q(\zeta)f(\zeta)}{g(\zeta)}}_\eG\; 
                               d\sigma(\zeta)
=
    \sum_{k,j=0}^\infty {\ip{Q_{j-k}f_k}{g_j}}_\eG ,
\end{equation}
where $\sigma$ is normalized Lebesgue measure on $\thecircle$ and
$f(\zeta) = f_0 + f_1 \zeta + f_2 \zeta^2 + \cdots$ and $g(\zeta) =
g_0 + g_1 \zeta + g_2 \zeta^2 + \cdots$ have coefficients in $\eG$,
all but finitely many of which are zero.  The operator $T_Q$ is
Toeplitz, and $T_P$ is analytic.  Moreover,
\begin{enumerate}
\item[$\bullet$] $Q(\zeta) \ge 0$ for all $\zeta\in\thecircle$
  if and only if $T_Q \ge 0$;
\smallskip

\item[$\bullet$] $Q(\zeta) = P(\zeta)^*P(\zeta)$ for all $\zeta
  \in\thecircle$ if and only if $T_Q = T_P^*T_P$.
\end{enumerate}

\begin{definition}\label{sep30b}
  We say that the polynomial $P(z)$ is {\bf outer} if the analytic
  Toeplitz operator $A=T_P$ is outer.
\end{definition}

In view of the example \eqref{aug31a}, the main problem is to write a
given nonnegative Toeplitz operator $T$ in the form $T=A^*A$, where
$A$ is analytic.  We also want to know that if $T=T_Q$ for a Laurent
polynomial~$Q$, then we can choose $A = T_P$ for an outer analytic
polynomial of the same degree.  Lemmas \ref{sep22a} and \ref{sep26k}
reduce the problem to showing that a certain isometry is a shift
operator.

\begin{lemma}[Lowdenslager's Criterion]\label{sep22a}
  Let $\eH$ be a Hilbert space, and let $S\in\lh$ be a shift operator.
  Let $T\in\lh$ be Toeplitz relative to $S$ as defined above, and
  suppose that $T \ge 0$.  Let $\eH_T$ be the closure of the range of
  $T^{1/2}$ in the inner product of $\eH$.  Then there is an isometry
  $S_T$ mapping $\eH_T$ into itself such that
\begin{equation*}
  S_T T^{1/2} f = T^{1/2}Sf, \qquad f \in\eH.
\end{equation*}
In order that $T=A^*A$ for some analytic operator $A \in \lh$, it is
necessary and sufficient that $S_T$ is a shift operator.  In this
case, $A$ can be chosen to be outer.
\end{lemma}

\begin{proof}
  The existence of the isometry $S_T$ follows from the identity
  $S^*TS=T$, which implies  that $T^{1/2}Sf$ and $T^{1/2} f$ have
  the same norms for any $f\in\eH$.
  
  If $S_T$ is a shift operator, we can view $\eH_T$ as a direct sum
  $\eH_T = \eG_T \oplus \eG_T \oplus \cdots$ with $S_T (h_0,h_1,\dots)
  = (0,h_0,h_1,\dots)$.  Here $\dim \eG_T \le \dim \eG$.  To see this,
  notice that a short argument shows that $T^{1/2}S_T^*$ and $S^*
  T^{1/2}$ agree on $\eH_T$, and therefore $T^{1/2} (\ker S_T^* )
  \subseteq \ker S^*$.  The dimension inequality then follows because
  $T^{1/2}$ is one-to-one on the closure of its range.  Therefore we
  may choose an isometry $V$ from $\eG_T$ into $\eG$.  Define an
  isometry $W$ on $\eH_T$ into $\eH$ by
\begin{equation*}
  W(h_0,h_1,\dots) = (Vh_0,Vh_1,\dots).
\end{equation*}
Define $A \in\lh$ by mapping $\eH$ into $\eH_T$ via $T^{1/2}$ and then
$\eH_T$ into $\eH$ via $W$:
\begin{equation*}
  Af = WT^{1/2}f, \qquad f\in\eH.
\end{equation*}
Straightforward arguments show that $A$ is analytic, outer, and
$T=A^*A$.

Conversely, suppose that $T=A^*A$ where $A\in\lh$ is analytic.  Define
an isometry $W$ on $\eH_T$ into $\eH$ by $WT^{1/2}f = Af$, $f \in\eH$.
Then $WS_T = SW$, and hence ${S_T^*}^n = W^*{S^*}^n W$ for all $n \ge
1$.  Since the powers of $S^*$ tend strongly to zero, so do the powers
of $S_T^*$, and therefore $S_T$ is a shift operator.
\end{proof}

\begin{lemma}\label{sep26k}
  In Lemma $\ref{sep22a}$, let $T=T_Q$ be given by $\eqref{aug31a}$
  for a Laurent polynomial $Q(z)$ of degree $m$.  If $T=A^*A$ where
  $A\in\lh$ is analytic and outer, then $A= T_P$ for some outer
  analytic polynomial $P(z)$ of degree $m$.
\end{lemma}

\begin{proof}
  Let $Q(z) = \sum_{k=-m}^m Q_k z^k$.  Recall that $Q_j = 0$ for
  $|j|>m$.  By \eqref{aug28b} applied to $A$, what we must show is
  that ${S^*}^j AE_0 = 0$ for all $j>m$.  It is sufficient to show
  that ${S^*}^{m+1} AE_0 = 0$.  By \eqref{aug28b} applied to $T$,
  since $T=A^*A$ and $A$ is analytic,
  \begin{equation*}
    E_0^*A^*{S^*}^j AE_0 = E_0^*{S^*}^j TE_0 = Q_j = 0,
    \qquad j>m.
  \end{equation*}
  It follows that $\ran {S^*}^{m+1} AE_0 \perp \ran AS^kE_0$ for all
  $k \ge 0$, and therefore
  \begin{equation}\label{aug31c}
    \ran {S^*}^{m+1} AE_0 \perp \ran A .
  \end{equation}
  Since $A$ is outer, $\ranclosure A$ reduces $S$, and so $\ran
  {S^*}^{m+1} AE_0 \subseteq \ranclosure A$.  Therefore
${S^*}^{m+1} AE_0 = 0$ by \eqref{aug31c}, and the result follows.
\end{proof}

The proof of the operator \fr theorem is now easily completed.

\begin{proof}[Proof of Theorem $\ref{ofr}$]
  Define $T = T_Q$ as in \eqref{aug31a}.  Lemmas \ref{sep22a} and
  \ref{sep26k} reduce the problem to showing that the isometry $S_T$
  is a shift operator.  It is sufficient to show that $\| {S_T^*}^n f
  \| \to 0$ for every $f$ in $\eH_T$.

  \noindent {\bf Claim:} If $f = T^{1/2}h$ where $h \in \eH$ has the 
  form $h = (h_0,\dots,h_r,0,\dots)$, then ${S_T^*}^n f = 0$ for all
  sufficiently large~$n$.

  For if $u \in\eH$ and $n$ is any positive integer, then
  \begin{equation*}
    {\ip{{S_T^*}^n f}{T^{1/2}u}}_{\eH_T}
=
    {\ip{f}{S_T^n T^{1/2}u}}_{\eH_T} 
=
    {\ip{T^{1/2}h}{T^{1/2}S^n u}}_{\eH}
=
    {\ip{Th}{S^n u}}_{\eH} .
  \end{equation*}
  By the definition of $T = T_Q$, $Th$ has only a finite number of
  nonzero entries (depending on $m$ and $r$), and the first $n$
  entries of $S^n u$ are zero (irrespective of $u$).  The claim
  follows from the arbitrariness of~$u$.
  
  In view of the claim, $\| {S_T^*}^n f \| \to 0$ for a dense set of
  vectors in $\eH_T$, and hence by approximation this holds for all
  $f$ in $\eH_T$. Thus $S_T$ is a shift operator, and the result
  follows.
\end{proof}

A more general result is proved in the original version of Theorem
\ref{ofr} in \cite{Rosenblum1968}.  There it is only required that
$Q(z)g$ is a Laurent polynomial for a dense set of $g$ in $\eG$ (the
degrees of these polynomials can be unbounded).  We have omitted an
accompanying uniqueness statement: the outer polynomial $P(z)$ in
Theorem \ref{ofr} can be chosen such that $P(0) \ge 0$, and then it is
unique.  See \cite{BC1992} and \cite{RRbook}.

\goodbreak
\section{Method of Schur complements}\label{S:Schur}

We outline now a completely different proof of the operator \fr
theorem.  The proof is due to Dritschel and Woerdeman \cite{DW2005}
and is based on the notion of a Schur complement.  The procedure
constructs the outer polynomial $P(z) = P_0 + P_1 z + \cdots + P_m
z^m$ one coefficient at a time.  A somewhat different use of Schur
complements in the operator \fr theorem appears in Dritschel
\cite{MAD2004}.  The method in \cite{MAD2004} plays a role in the
multivariable theory, which is taken up in \S\ref{S:multivar}.

We shall explain the main steps of the construction assuming the
validity of two lemmas.  Full details are given in \cite{DW2005} and
also in the forthcoming book \cite{BWbook} by Bakonyi and Woerdeman.
The authors thank Mihaly Bakonyi and Hugo Woerdeman for advance copies
of key parts of \cite{BWbook}, which has been helpful for our
exposition.  The book \cite{BWbook} includes many additional results
not discussed here.

\begin{definition}\label{nov4e}
  Let $\eH$ be a Hilbert space.  Suppose $T\in\lh$, $T\ge 0$.  Let
  $\eK$ be a closed subspace of $\eH$, and let $P_\eK \in\lhk$ be
  orthogonal projection of $\eH$ onto $\eK$.  Then (see
  Appendix~\ref{S:complement}, Lemma \ref{sep26c}) there is a unique
  operator $S\in\lk$, $S \ge 0$, such that
  \begin{enumerate}
  \item[(i)] $T-P_\eK^* S P_\eK \ge 0$;
  \item[(ii)] if $\widetilde S \in\lk$, $\widetilde S \ge 0$, and
    $T-P_\eK^* \widetilde S P_\eK \ge 0$, then $\widetilde S \le S$.
  \end{enumerate}
  We write $S = S(T,\eK)$ and call $S$ the {\bf Schur complement of
    $T$ supported on $\eK$}.
 \end{definition}
 
 Schur complements satisfy an inheritance property, namely, if $\eK_-
 \subseteq \eK_+ \subseteq \eH$, then $S(T,\eK_-) =
 S(S(T,\eK_+),\eK_-)$.  If $T$ is specified in matrix form,
\begin{equation*}
  T = 
  \begin{pmatrix}
    A & B^* \\ B & C
  \end{pmatrix} \colon \eK \oplus \eK^\perp \to \eK \oplus \eK^\perp ,
\end{equation*}
then $S = S(T,\eK)$ is the largest nonnegative operator in $\lk$ such
that
  \begin{equation*}
    \begin{pmatrix}
    A-S & B^* \\ B & C
  \end{pmatrix} \ge 0.
  \end{equation*}
  The condition $T \ge 0$ is equivalent to the existence of a
  contraction $G \in \eL(\eK,\eK^\perp)$ such that $B = C^\frac12 G
  A^\frac12$ (Appendix~\ref{S:complement}, Lemma \ref{sep27b}).  In
  this case, $G$ can be chosen so that it maps $\ranclosure A$ into
  $\ranclosure C$ and is zero on the orthogonal complement of
  $\ranclosure A$, and then
\begin{equation*}
  S = A^\frac12 (I-G^*G) A^\frac12 .
\end{equation*}
When $C$ is invertible, this reduces to the familiar formula $S = A -
B^*C^{-1}B$.

  \begin{lemma}\label{lemma2}
    Let $M\in\lh$, $M\ge 0$, and suppose that
    \begin{equation*}
      M = \begin{pmatrix}
    A & B^* \\ B & C
  \end{pmatrix} \colon \eK \oplus \eK^\perp \to \eK \oplus \eK^\perp
    \end{equation*}
    for some closed subspace $\eK$ of $\eH$.  

\noindent $(1)$  If $S(M,\eK)=P^*P$ and
    $C=R^*R$ for some $P\in \lk$ and $R\in\lkperp$, then there is a
    unique $X \in \eL(\eK,\eK^\perp)$ such that
      \begin{equation}\label{oct10b}
        M = 
        \begin{pmatrix}
          P^*&X^* \\ 0&R^*
        \end{pmatrix}
        \begin{pmatrix}
          P&0 \\ X&R
        \end{pmatrix}  
\qquad\text{and}\qquad
        \ran X \subseteq \ranclosure R.
      \end{equation}

\noindent $(2)$
      Conversely, if $\eqref{oct10b}$ holds for some operators
      $P,X,R$, then $S(M,\eK)=P^*P$.
  \end{lemma}

We omit the proof and refer the reader to \cite{BWbook} or
\cite{DW2005} for details.

\goodbreak
  \begin{proof}[Proof of Theorem $\ref{ofr}$ using Schur complements]
    Let $Q(z) = \sum_{k=-m}^m Q_k z^k$ satisfy $Q(\zeta) \ge 0$ for all
    $\zeta\in\thecircle$.  We shall recursively construct the coefficients
    of an outer polynomial $P(z) = P_0 + P_1 z + \cdots + P_m z^m$
    such that $Q(\zeta) = P(\zeta)^*P(\zeta)$, $\zeta\in\thecircle$.

    Write $\eH = \eG \oplus \eG \oplus \cdots$ and $\eG^n = \eG \oplus
    \cdots \oplus \eG$ with $n$ summands.  As before, set $Q_k=0$ for
    $|k|>m$, and define $T_Q \in\lh$ by
    \begin{equation*}
      T_Q = \begin{pmatrix}
    Q_0 & Q_{-1} & Q_{-2} & \cdots\\
    Q_1 & Q_0    & Q_{-1} & \ddots \\
    Q_2 & Q_1  & Q_0 & \ddots \\
    \vdots    & \ddots  & \ddots & \ddots
  \end{pmatrix} .
    \end{equation*}
    For each $k=0,1,2,\dots$, define
    \begin{equation*}
      S(k) = S(T_Q,\eG^{k+1}),
    \end{equation*}
    which we interpret as the Schur complement of $T_Q$ on the first
    $k+1$ summands of $\eH = \eG \oplus \eG \oplus \cdots$.  Thus
    $S(k)$ is a $(k+1)\times (k+1)$ block operator matrix satisfying
    \begin{equation}\label{nov3a}
      S(S(k),\eG^{j+1}) = S(j), \quad 0 \le j < k < \infty
    \end{equation}
    by the inheritance property of Schur complements.

    \begin{lemma}\label{lemma3}
      For each $k=0,1,2,\dots$,
      \begin{equation*}
        S(k+1) = 
        \begin{pmatrix}
          Y_0 & 
              \begin{pmatrix} Y_1 & \cdots & Y_{k+1} \end{pmatrix} \\[3pt]
          \begin{pmatrix} Y_1^* \\ \vdots \\ Y_{k+1}^* \end{pmatrix} 
                & S(k)
        \end{pmatrix}
      \end{equation*}
      for some operators $Y_0,Y_1,\dots,Y_{k+1}$ in $\lg$.  For $k \ge
      m-1$,
      \begin{equation*}
        \begin{pmatrix} Y_0 & Y_1 & \cdots & Y_{k+1} \end{pmatrix}
        =
        \begin{pmatrix} Q_0 & Q_{-1} & 
                       \cdots & Q_{-k-1} \end{pmatrix} .        
      \end{equation*}
    \end{lemma}
    
    Again see \cite{BWbook} or \cite{DW2005} for details.  Granting
    Lemmas \ref{lemma2} and \ref{lemma3}, we can proceed with the
    construction.  \smallskip

    \noindent {\bf Construction of $P_0, P_1$.}  Choose $P_0 =
    S(0)^\frac12$.  Using Lemma \ref{lemma3}, write
    \begin{equation*}
      S(1) = 
      \begin{pmatrix}
        Y_0 & Y_1 \\ Y_1^* & S(0) 
      \end{pmatrix} .
    \end{equation*}
    In Lemma \ref{lemma2}(1) take $M = S(1)$ and use the
    factorizations
    \begin{equation*}
      S(S(1),\eG^1) \stackrel{\eqref{nov3a}}{=} S(0) = P_0^*P_0 
      \qquad\text{and}\qquad
      S(0) = P_0^*P_0 \,.
    \end{equation*}
    Choose $P_1=X$ where $X\in\lg$ is the operator produced by Lemma
    \ref{lemma2}(1).  Then
    \begin{equation}
      \label{oct10c}
        S(1) = 
      \begin{pmatrix}
        P_0^*&P_1^* \\ 0&P_0^* 
      \end{pmatrix} 
      \begin{pmatrix}
        P_0&0 \\ P_1&P_0
      \end{pmatrix}
\qquad\text{and}\qquad
       \ran P_1 \subseteq \ranclosure P_0 .
    \end{equation}

\smallskip

    \noindent {\bf Construction of $P_2$.}  Next use Lemma
    \ref{lemma3} to write
    \begin{equation*}
      S(2) = 
      \begin{pmatrix}
        Y_0 & 
        \begin{pmatrix}
          Y_1&Y_2
        \end{pmatrix}  \\[3pt]
\begin{pmatrix}
  Y_1^* \\ Y_2^*
\end{pmatrix} & S(1) 
      \end{pmatrix} ,
    \end{equation*}
    and apply Lemma \ref{lemma2}(1) to $M = S(2)$ with the factorizations
    \begin{gather*}
      S(S(2),\eG^1) \stackrel{\eqref{nov3a}}{=} S(0) = P_0^*P_0, \\
S(1) =  \begin{pmatrix}
        P_0^*&P_1^* \\ 0&P_0^* 
      \end{pmatrix} 
      \begin{pmatrix}
        P_0&0 \\ P_1&P_0
      \end{pmatrix} .
    \end{gather*}
    This yields operators $X_1,X_2\in\lg$ such that
    \begin{gather}\label{oct10d}
      S(2) = 
        \begin{pmatrix}
          P_0^*&X_1^*&X_2^* \\
          0    &P_0^*&P_1^* \\
          0    & 0   &P_0^*
        \end{pmatrix}
        \begin{pmatrix}
          P_0&0&0 \\
          X_1 &P_0 &0 \\
          X_2& P_1&P_0
        \end{pmatrix} , \\
     \ran 
     \begin{pmatrix}
       X_1 \\ X_2
     \end{pmatrix} \subseteq 
     \ranclosure \begin{pmatrix}
       P_0&0 \\ P_1&P_0
     \end{pmatrix} .  \label{oct10e}
    \end{gather}
    In fact, $X_1=P_1$.  To see this, notice that we can rewrite
    \eqref{oct10d} as
  \begin{gather*}
    S(2) = 
    \begin{pmatrix}
      {\widetilde P}^* & {\widetilde X}^* \\ 0 & {\widetilde R}^*
    \end{pmatrix}
    \begin{pmatrix}
      \widetilde P & 0 \\ \widetilde X & \widetilde R
    \end{pmatrix} , \\
\widetilde P = \begin{pmatrix} P_0 & 0 \\ X_1 & P_0 \end{pmatrix},
\qquad
\widetilde X = \begin{pmatrix} X_2 & P_1 \end{pmatrix},
\qquad
\widetilde R = P_0.
  \end{gather*}
  By \eqref{oct10c} and \eqref{oct10e}, $\ran P_1 \subseteq
  \ranclosure P_0$ and $\ran X_2 \subseteq \ranclosure P_0$, and
  therefore $\ran \widetilde X \subseteq \ranclosure P_0$.  Hence by
  Lemma \ref{lemma2}(2), 
  \begin{equation}
    \label{nov2b}
    S(S(2),\eG^2) 
         = {\widetilde P}^* \widetilde P 
         = 
         \begin{pmatrix}
           P_0^* & X_1^* \\ 0 & P_0^*
         \end{pmatrix}
         \begin{pmatrix}
           P_0 & 0 \\ X_1 & P_0 .
         \end{pmatrix} .
  \end{equation}
  Comparing this with
  \begin{equation}\label{nov2a}
    S(S(2),\eG^2)\stackrel{\eqref{nov3a}}{=}  S(1)
       \stackrel{\eqref{oct10c}}{=}
\begin{pmatrix}
        P_0^*&P_1^* \\ 0&P_0^* 
      \end{pmatrix} 
      \begin{pmatrix}
        P_0&0 \\ P_1&P_0
      \end{pmatrix} ,
  \end{equation}
  we get $P_0^*P_1 = P_0^*X_1$.  By \eqref{oct10e}, $\ran X_1
  \subseteq \ranclosure P_0$, and therefore $X_1 = P_1$.  Now choose
  $P_2 = X_2$ to obtain
    \begin{gather}\label{oct10dd}
      S(2) = 
        \begin{pmatrix}
          P_0^*&P_1^*&P_2^* \\
          0    &P_0^*&P_1^* \\
          0    & 0   &P_0^*
        \end{pmatrix}
        \begin{pmatrix}
          P_0&0&0 \\
          P_1 &P_0 &0 \\
          P_2& P_1&P_0
        \end{pmatrix} , \\
     \ran 
     \begin{pmatrix}
       P_1 \\ P_2
     \end{pmatrix} \subseteq 
     \ranclosure \begin{pmatrix}
       P_0&0 \\ P_1&P_0
     \end{pmatrix} .  \label{oct10ee}
    \end{gather}

\noindent {\bf Inductive step.}
We continue in the same way for all $k=1,2,3,\dots$.  At the $k$-th
stage, the procedure produces operators $P_0,\dots,P_k$ such that 
\begin{gather}
  S(k) = 
  \begin{pmatrix}
    P_0^* & \cdots & P_k^*  \\
          & \ddots & \vdots \\
      0    && P_0^*
  \end{pmatrix}
  \begin{pmatrix}
    P_0 && 0 \\
    \vdots & \ddots & \\
    P_k & \cdots & P_0
  \end{pmatrix} ,
\label{nov4a}
\\
\ran 
\begin{pmatrix}
  P_1 \\ \vdots \\ P_k
\end{pmatrix}
\subseteq   \ranclosure
   \begin{pmatrix}
    P_0 && 0 \\
    \vdots & \ddots & \\
    P_k & \cdots & P_0
  \end{pmatrix} .
\label{nov4b}
\end{gather}
By Lemma \ref{lemma3}, in the case $k \ge m$, 
\begin{equation}\label{nov4c}
  S(k) = 
  \begin{pmatrix}
    Q_0 & \begin{pmatrix} Q_{-1} & \cdots & Q_{-m} 
                    &  0& \cdots &0 \end{pmatrix} \\[3pt]
    \begin{pmatrix} Q_{-1}^* \\ \cdots \\ Q_{-m}^* 
                    \\  0\\ \vdots \\ 0 \end{pmatrix}
  &S(k-1)
  \end{pmatrix} .
\end{equation}
The zeros appear here when $k>m$, and their presence leads to the
conclusion that $P_k = 0$ for $k>m$.  We set then
\begin{equation*}
  P(z) = P_0 + P_1 z + \cdots + P_m z^m .
\end{equation*}
Comparing \eqref{nov4a} and \eqref{nov4c} in the case $k=m$, we deduce
$2m+1$ relations which are equivalent to the identity
\begin{equation*}
  Q(\zeta) = P(\zeta)^*P(\zeta), \qquad \zeta \in\thecircle .
\end{equation*}

\noindent {\bf Final step: $P(z)$ is outer.}
Define $T_P$ as in \eqref{aug31a}.  With natural identifications,
\begin{equation}
  \label{nov4d}
  T_P = 
  \begin{pmatrix}
    P_0 & 
    \begin{pmatrix}0&0&\cdots&  \end{pmatrix}\, \\[3pt]
    \begin{pmatrix} P_1 \\ P_2 \\ \vdots \\  \end{pmatrix}
           & T_P
  \end{pmatrix} .
\end{equation}
The relations \eqref{nov4b}, combined with the fact that $P_k=0$ for
all $k>m$, imply that 
\begin{equation*}
  \ran \begin{pmatrix} P_1 \\ P_2 \\ \vdots \\  \end{pmatrix}
\subseteq
  \ranclosure\, T_P .
\end{equation*}
Hence for any $g\in\eG$, a sequence $f_n$ can be found such that 
\begin{equation*}
  T_P f_n \to \begin{pmatrix} P_1 \\ P_2 \\ \vdots \\  \end{pmatrix}
  g.
\end{equation*}
Then by \eqref{nov4d},
\begin{equation*}
  T_P 
  \begin{pmatrix}
    g \\ f_n
  \end{pmatrix} \to 
  \begin{pmatrix}
    P_0 \\ 0
  \end{pmatrix} g .
\end{equation*}
It follows that $\ranclosure\, T_P$ contains every vector $(P_0
g,0,0,\dots )$ with $g\in\lg$, and hence $\ranclosure\, T_P \supseteq
\ranclosure P_0 \oplus \ranclosure P_0 \oplus \cdots$.  The reverse
inclusion holds because by \eqref{nov4b}, the ranges of
$P_1,P_2,\dots$ are all contained in $\ranclosure P_0$.  Thus $P(z)$
is outer.
  \end{proof}

\goodbreak
\section{Spectral factorization}\label{S:spectral}

The problem of spectral factorization is to write a nonnegative
operator-valued function $F$ on the unit circle in the form $F = G^*G$
where $G$ is analytic (in a sense made precise below).  The
terminology comes from prediction theory, where the nonnegative
function $F$ plays the role of a spectral density for a
multidimensional stationary stochastic process.
The problem may be viewed as a generalization of a classical theorem
of Szeg{\H o} from Hardy class theory and the theory of orthogonal
polynomials (see Hoffman \cite[p.\ 56]{hoffman} and Szeg{\H o}
\cite[Chapter X]{szego}).

We write $H^p$ and $L^p$ for the standard Hardy and Lebesgue spaces
for the unit disk and unit circle.  See Duren \cite{duren1970}.
Recall that $\sigma$ is normalized Lebesgue measure on the unit
circle~$\thecircle$.

\begin{theorem}[Szeg{\H o}'s Theorem]\label{nov23d}
  Let $w \in L^1$ satisfy $w \ge 0$ a.e.\ on $\thecircle$ and
  \begin{equation*}
    \int_{\thecircle} \log w(\zeta)\; d\sigma >-\infty .
  \end{equation*}
  Then $w = |h|^2$ a.e.\ on $\thecircle$ for some $h \in H^2$, and $h$
  can be chosen to be an outer function.
\end{theorem}

Operator and matrix generalizations of Szeg{\H o}'s theorem are stated
in Theorems \ref{nov23c} and \ref{nov23cc} below.  Some vectorial
function theory is needed to formulate these and other results.  We
assume familiarity with basic concepts but recall a few definitions.
For details, see e.g.\ \cite{Helson1964, NFbook}) and \cite[Chapter
4]{RRbook}.

In this section, $\eG$ denotes a separable Hilbert space.  Functions
$f$ and $F$ on the unit circle with values in $\eG$ and $\lg$,
respectively, are called weakly measurable if $\ip{f(\zeta)}{v}$ and
$\ip{F(\zeta)u}{v}$ are measurable for all $u,v \in\eG$.
Nontangential limits for analytic functions on the unit disk are taken
in the strong (norm) topology for vector-valued functions, and in the
strong operator topology for operator-valued functions.  We fix
notation as follows:
\begin{enumerate}
\item[(i)] We write \smash{$L^2_\eG$} and
\smash{$L^\infty_{\lg}$} for the standard Lebesgue spaces
of weakly measurable functions on the unit circle with values in $\eG$
and $\lg$.
\smallskip

\item[(ii)] Let $H^2_\eG$ and $H^\infty_{\lg}$ be the analogous 
Hardy classes of analytic functions on the unit disk.  We identify
elements of these spaces with their nontangential boundary functions,
and so the spaces may alternatively be viewed as subspaces of
\smash{$L^2_\eG$} and
\smash{$L^\infty_{\lg}$}. 
\smallskip

\item[(iii)] Let \smash{$N^+_{\lg}$} be the space of all analytic
  functions $F$ on the unit disk such that $\varphi F$ belongs to
  $H^\infty_{\lg}$ for some bounded scalar outer function~$\varphi$.
  The elements of \smash{$N^+_{\lg}$} are also identified with their
  nontangential boundary functions.

\end{enumerate}

A function $F \in H^\infty_{\lg}$ is called {\bf outer} if $FH^2_\eG$
is dense in $H^2_\eF$ for some closed subspace $\eF$ of $\eG$.  A
function $F\in N^+_{\lg}$ is {\bf outer} if there is a bounded scalar
outer function $\varphi$ such that $\varphi F \in H^\infty_{\lg}$ and
$\varphi F$ is outer in the sense just defined.  The definition of an
outer function given here is consistent with the previously defined
notion for polynomials in \S\ref{S:FejerRiesz}.

A function $A \in H^\infty_{\lg}$ is called {\bf inner} if
multiplication by $A$ on $\eH^2_\eG$ is a partial isometry.  In this
case, the initial space of multiplication by $A$ is a subspace of
$\eH^2_\eG$ of the form $\eH^2_\eF$ where $\eF$ is a closed subspace
of $\eG$.  To prove this, notice that both the kernel of
multiplication by $A$ and the set on which it is isometric are
invariant under multiplication by $z$.  Therefore the initial space of
multiplication by $A$ is a reducing subspace for multiplication by
$z$, and so it has the form $\eH^2_\eF$ where $\eF$ is a closed
subspace of $\eG$ (see \cite[p.\ 106]{halmos1961} and \cite[p.\ 
p.96]{RRbook}).

Every $F \in H^\infty_{\lg}$ has an {\bf inner-outer factorization} $F
= AG$, where $A$ is an inner function and $G$ is an outer function.
This factorization can be chosen such that the isometric set $H^2_\eF$
for multiplication by $A$ on $H^2_\eG$ coincides with the closure of
the range of multiplication by $G$.  The inner-outer factorization is
extended in an obvious way to functions $F\in N^+_{\lg}$.  Details are
given, for example, in \cite[Chapter 5]{RRbook}.

The main problem of this section can now be interpreted more
precisely:
\smallskip

\noindent {\bf Factorization Problem.}  {\it Given a nonnegative 
  weakly measurable function $F$ on $\thecircle$, find a function $G$
  in $N^+_{\lg}$ such that $F = G^*G$ a.e.\ on $\thecircle$.  If such
  a function exists, we say that $F$ is {\bf factorable}}.  \smallskip

If a factorization exists, the factor $G$ can be chosen to be outer by
the inner-outer factorization.  Moreover, an outer factor $G$ can be
chosen such that $G(0) \ge 0$, and then it is unique \cite[p.\ 
101]{RRbook}.  By the definition of $N^+_{\lg}$, a necessary condition
for $F$ to be factorable is that
\begin{equation}\label{nov20a}
  \int_\thecircle \log^+\| F(\zeta) \| \; d\sigma < \infty,
\end{equation}
where $\log^+ x$ is zero or $\log x$ according as $0 \le x \le 1$ or
$1 < x < \infty$, and so we only need consider functions which satisfy
\eqref{nov20a}.  In fact, in proofs we can usually reduce to the
bounded case by considering $F/|\varphi|^2$ for a suitable scalar
outer function~$\varphi$.

The following result is another view of Lowdenslager's criterion,
which we deduce from Lemma \ref{sep22a}.  A direct proof is given in
\cite[pp.\ 201--203]{NFbook}.

\begin{lemma}\label{nov15a}
  Suppose \smash{$F\in L^\infty_{\lg}$} and $F \ge 0$ a.e.\ on
  $\thecircle$.  Let $\eK_F$ be the closure of $F^\frac12 H^2_\eG$ in
  $L^2_\eG$, and let $S_F$ be the isometry multiplication by $\zeta$
  on $\eK_F$.  Then $F$ is factorable if and only if $S_F$ is a shift
  operator, that is,
  \begin{equation}\label{nov15c}
    \bigcap_{n=0}^\infty \zeta^n\;\overline{ F^\frac12
             H^2_\eG} = \{0\} .
  \end{equation}
\end{lemma}

\begin{proof}
  In Lemma \ref{sep22a} take $\eH = H^2_\eG$ viewed as a subspace of
  $L^2_\eG$, and let $S$ be multiplication by $\zeta$ on $\eH$.
  Define $T\in\lh$ by $Tf = PFf$, $f\in\eH$, where $P$ is the
  projection from $L^2_\eG$ onto $H^2_\eG$.  One sees easily that $T$
  is a nonnegative Toeplitz operator, and so we can define $\eH_T$ and
  an isometry $S_T$ as in Lemma \ref{sep22a}.  In fact, $S_T$ is
  unitarily equivalent to $S_F$ via the natural isomorphism $W \colon
  \eH_T \to \eK_F$ such that $W (T^\frac12 f) = F^\frac12 f$ for every
  $f$ in $\eH$.  Thus $S_F$ is a shift operator if and only if $S_T$
  is a shift operator, and by Lemma \ref{sep22a} this is the same as
  saying that $T = A^*A$ where $A \in\lh$ is analytic, or equivalently
  $F$ is factorable \cite[p.\ 110]{RRbook}.
\end{proof}

We obtain a very useful sufficient condition for factorability.

\begin{theorem}\label{nov20c}
  Suppose $F \in L^\infty_{\lg}$ and $F \ge 0$ a.e.  For $F$ to be
  factorable, it is sufficient that there exists a function $\psi$ in
  $L^\infty_{\lg}$ such that
  \begin{enumerate}
  \item[(i)] $\psi F \in H^\infty_{\lg}$;
\item[(ii)] for all $\zeta \in\thecircle$ except at most a set of
  measure zero, $\psi(\zeta) \big\vert \overline{F(\zeta)\eG}$ is
  one-to-one.
  \end{enumerate}
  If these conditions are met and $F=G^*G$ a.e.\ with $G$ outer,
  then $\psi G^* \in H^\infty_{\lg}$.
\end{theorem}

Theorem \ref{nov20c} appears in Rosenblum \cite{Rosenblum1968} with
$\psi(\zeta)=\zeta^m$ (viewed as an operator-valued function).  The
case of an arbitrary inner function was proved and applied in a
variety of ways by Rosenblum and Rovnyak \cite{RRbook, RR1971}.
V.~I.\ Matsaev first showed that more general functions $\psi$ can be
used.  Matsaev's result is evidently unpublished, but versions were
given by D.~Z.\ Arov \cite[Lemma to Theorem 4]{Arov1979} and A.~S.\ 
Markus \cite[Theorem 34.3 on p.\ 199]{Markus}.  Theorem \ref{nov20c}
includes all of these versions.

We do not know if the conditions (i) and (ii) in Theorem \ref{nov20c}
are necessary for factorability.  It is not hard to see that they are
necessary in the simple cases $\dim \eG =1$ and $\dim \eG =2$ (for the
latter case, one can use \cite[Example~1, p.\ 125]{RRbook}).  The
general case, however, is open.

\begin{proof}[Proof of Theorem $\ref{nov20c}$]
  Let $F$ satisfy (i) and (ii).  Define a subspace $\eM$ of
  $L^2_\eG$ by
  \begin{equation*}
    \eM = \bigcap_{n=0}^\infty \zeta^n\;\overline{ F^\frac12
             H^2_\eG}
        = \bigcap_{n=0}^\infty \overline{\zeta^n F^\frac12
    H^2_\eG} .
  \end{equation*}
  We show that $\eM = \{0\}$.  By (i),
  \begin{equation}\label{jan22a}
    \psi F^\frac12 \eM 
= 
     \psi F^\frac12 \bigcap_{n=0}^\infty 
                         \overline{\zeta^n F^\frac12 H^2_\eG}
\subseteq
   \bigcap_{n=0}^\infty \overline{\zeta^n \psi F H^2_\eG}
\subseteq
         \bigcap_{n=0}^\infty \zeta^n H^2_\eG
=
     \{0\} .
  \end{equation}
  Thus $\psi F^\frac12 \eM = \{0\}$.  Now if $g \in \eM$, then $\psi
  F^\frac12 g = 0$ a.e.\ by \eqref{jan22a}.  Hence $F^\frac12 g = 0$
  a.e.\ by (ii).  By the definition of $\eM$, $g \in
  \overline{F^\frac12 H^2_\eG}$, and standard arguments show from this
  that $g(\zeta) \in \overline{F(\zeta)^\frac12 \eG}$ a.e.  Therefore
  $g = 0$ a.e.  It follows that $\eM = \{0\}$, and so $F$ is
  factorable by Lemma \ref{nov15a}.
  
  Let $F=G^*G$ a.e.\ with $G$ outer.  We prove that $\psi G^* \in
  H^\infty_{\lg}$ by showing that $\psi G^* H^2_\eG \subseteq
  H^2_\eG$.  Since $G$ is outer, $\overline{G H^2_\eG} = H^2_\eF$ for
  some closed subspace $\eF$ of~$\eG$.  By (i),
  \begin{equation*}
    \psi G^* (G H^2_\eG) = \psi F H^2_\eG \subseteq H^2_\eG.
  \end{equation*}
  Therefore $\psi G^* H^2_\eF \subseteq H^2_\eG$.  Suppose $f \in
  H^2_{\eG\ominus\eF}$, and consider any $h\in L^2_\eG$.  Then
  \begin{equation*}
    {\ip{G^*f}{h}}_{L^2_\eG} = \int_\thecircle
    {\ip{f(\zeta)}{G(\zeta)h(\zeta)}}_\eG \; d\sigma = 0 ,
  \end{equation*}
  because $\ran G(\zeta) \subseteq \eF$ a.e.  Thus $\psi G^* f = 0$
  a.e.  It follows that $\psi G^* H^2_\eG \subseteq H^2_\eG$, and
  therefore $\psi G^* \in H^\infty_{\lg}$.
\end{proof}

For a simple application of Theorem \ref{nov20c}, suppose that $F$ is
a Laurent polynomial of degree $m$, and choose $\psi$ to be $\zeta^m
I$.  In short order, this yields another proof of the operator \fr
theorem (Theorem \ref{ofr}).

Another application is a theorem of Sarason \cite[p.\ 
198]{Sarason1967}, which generalizes the factorization of a
scalar-valued function in $H^1$ as a product of two functions in $H^2$
(see \cite[p.\ 56]{hoffman}).
 
\begin{theorem}\label{nov20b}
  Every $G$ in $N^+_{\lg}$ can be written $G = G_1G_2$, where $G_1$
  and $G_2$ belong to $N^+_{\lg}$ and
  \begin{equation*}
    G_2^*G_2 = [G^*G]^{1/2}
\quad\text{and}\quad 
    G_1^*G_1 = G_2G_2^*
\qquad {\text{a.e.}}
  \end{equation*}
\end{theorem}

\begin{proof}
  Suppose first that $G \in H^\infty_{\lg}$.  For each $\zeta
  \in\thecircle$, write
  \begin{equation*}
    G(\zeta) = U(\zeta)[G^*(\zeta)G(\zeta)]^\frac12 ,
  \end{equation*}
  where $U(\zeta)$ is a partial isometry with initial space
  $\ranclosure [G^*(\zeta)G(\zeta)]^\frac12$.  It can be shown that
  $U$ is weakly measurable.  We apply Theorem \ref{nov20c} with $F =
  [G^*G]^\frac12$ and $\psi = U$.  Conditions (i) and (ii) of Theorem
  \ref{nov20c} are obviously satisfied, and so we obtain an outer
  function $G_2 \in H^\infty_{\lg}$ such that
  \begin{equation*}
        G_2^*G_2 = [G^*G]^{1/2} \qquad {\text{a.e.}}
  \end{equation*}
  and $UG_2^* \in H^\infty_{\lg}$.  Set $G_1 = UG_2^*$.  By
  construction $G_1 \in H^\infty_{\lg}$,
  \begin{equation*}
    G = U(G^*G)^\frac12 = (UG_2^*)G_2 = G_1G_2,
  \end{equation*}
  and $G_1^*G_1 = G_2U^*UG_2^* = G_2G_2^*$ a.e.  The result follows
  when $G \in H^\infty_{\lg}$.

  The general case follows on applying what has just been shown to
  $\varphi^2 G$, where $\varphi$ is a scalar-valued outer function such
  that $\varphi^2 G \in H^\infty_{\lg}$.
\end{proof}

The standard operator generalization of Szeg{\H o}'s theorem also
follows from Theorem \ref{nov20c}.

\begin{theorem}\label{nov23c}
  Let $F$ be a weakly measurable function on $\thecircle$ with values
  in $\lg$ satisfying $F \ge 0$ a.e.\ and
  \begin{equation*}
    \int_\thecircle \log^+ \| F(\zeta) \| \; d\sigma < \infty
\quad\text{and}\quad
    \int_\thecircle \log^+ \| F(\zeta)^{-1} \| \; d\sigma < \infty .
  \end{equation*}
  Then  $F$ is factorable.
\end{theorem}

\begin{proof}
  Since $\log^+ \|F(\zeta) \|$ is integrable, we can choose a scalar
  outer function $\varphi_1$ such that
  \begin{equation*}
    F_1 = F/|\varphi_1|^2 \in  L^\infty_{\lg}.
  \end{equation*}
  Since $\log^+ \|F(\zeta)^{-1} \|$ is integrable,
  so is $\log^+ \|F_1(\zeta)^{-1} \|$.  Hence there is a bounded
  scalar outer function $\varphi$ such that 
  \begin{equation*}
     \varphi F_1^{-1} \in  L^\infty_{\lg}.
  \end{equation*}
  We apply Theorem \ref{nov20c} to $F_1$ with $\psi = \varphi
  F_1^{-1}$.  Condition (i) is satisfied because $\psi F_1 = \varphi
  I$.  Condition (ii) holds because the values of $\psi$ are
  invertible a.e.  Thus $F_1$ is factorable, and hence so is $F$.
\end{proof}

Theorem \ref{nov20c} has a half-plane version, the scalar inner case
of which is given in \cite[p.\ 117]{RRbook}.  This has an application
to the following generalization of Akhiezer's theorem on factoring
entire functions \cite[Chapter 6]{RRbook}.

\begin{theorem}
  Let $F$ be an entire function of exponential type $\tau$, having
  values in $\lg$, such that $F(x) \ge 0$ for all real $x$ and
  \begin{equation*}
    \int_{-\infty}^\infty \frac{\log^+ \|F(t)\|}{1+t^2}\; dt
            < \infty.
  \end{equation*}
  Then $F(x) = G(x)^*G(x)$ for all real $x$ where $G$ is an entire
  function with values in $\lg$ such that $exp(-i\tau z/2) G$ is of
  exponential type $\tau/2$ and the restriction of $G$ to the upper
  half-plane is an outer function.
\end{theorem}

\smallskip
\noindent {\bf Matrix case:}
\medskip

We end this section by quoting a few results for matrix-valued
functions.  The matrix setting is more concrete, and one can do more.
Statements often require invertibility assumptions.  We give no
details and leave it to the interested reader to consult other sources
for further information.

Our previous definitions and results transfer in an obvious way to
matrix-valued functions.  For this we choose $\eG = \bbC^r$ for some
positive integer $r$ and indentify operators on $\bbC^r$ with $r
\times r$ matrices.  The operator norm of a matrix is denoted $\|
\cdot \|$.  We write $L^\infty_{r \times r}, H^\infty_{r \times r}$ in
place of $L^\infty_{\lg}, H^\infty_{\lg}$ and $\| \cdot \|_\infty$ for
the norms on these spaces.

Theorem \ref{nov23c} is more commonly stated in a
different form for matrix-valued functions.

\begin{theorem}\label{nov23cc}
  Suppose that $F$ is an $r \times r$ measurable matrix-valued
  function having invertible values on $\thecircle$ such that $F \ge
  0$ a.e.\ and $\log^+ \| F \|$ is integrable.  Then $F$ is
  factorable if and only if $\log \det F$ is integrable.
\end{theorem}

Recall that when $F$ is factorable, there is a unique outer $G$ such
that $F=G^*G$ and $G(0) \ge 0$.  It makes sense to inquire about the
continuity properties of the mapping $\Phi \colon F \to G$ with
respect to various norms.  For example, see Jacob and Partington
\cite{JP2001}.  We cite one recent result in this area.

\begin{theorem}[Barclay \cite{Barclay2004}]
  Let $F,F_n$, $n=1,2,\dots$, be $r \times r$ measurable matrix-valued
  functions on $\thecircle$ having invertible values a.e.\ and
  integrable norms.  Suppose that $F = G^*G$ and $F_n = G_n^*G_n$,
  where $G,G_n$ are $r \times r$ matrix-valued outer functions such
  that $G(0) \ge 0$ and $G_n(0) \ge 0$, $n = 1,2,\dots$.  Then
   \begin{equation*}
    \lim_{n \to \infty} \int_\thecircle
      \| G(\zeta) - G_n(\zeta) \|^2  \; d\sigma = 0
  \end{equation*}
  if and only if
  \begin{enumerate}
  \item[(i)] $\displaystyle{\lim_{n \to \infty} \int_\thecircle
      \| F(\zeta) - F_n(\zeta) \|  \; d\sigma = 0}$, and
\smallskip

  \item[(ii)] the family of functions $\{ \log \det F_n \}_{n=0}^\infty$
  is uniformly integrable.
  \end{enumerate}
\end{theorem}

A family of functions $\{ \varphi_\alpha \}_{\alpha \in A} \subseteq
L^1$ is {\bf uniformly integrable} if for every $\varepsilon > 0$
there is a $\delta > 0$ such that $\int_E |\varphi_\alpha| \, d\sigma
< \varepsilon$ for all $\alpha \in A$ whenever $\sigma(E) < \delta$.
See \cite{Barclay2004} for additional references and similar results
in other norms.

A theorem of Bourgain \cite{Bourgain1986} characterizes all functions
on the unit circle which are products $\bar h g$ with $g,h\in H^\infty$:
{\it A function $f \in L^\infty$ has the form $f = \bar h g$ where
  $g,h\in H^\infty$ if and only if $\log |f|$ is integrable.}  This
resolves a problem of Douglas and Rudin \cite{DouglasRudin}.  The
problem is more delicate than spectral factorization; when $|f|=1$
a.e., the factorization cannot be achieved in general with inner
functions.  Bourgain's theorem was recently generalized to
matrix-valued functions.

\begin{theorem}[Barclay \cite{BarclayPreprint, 
                             BarclayThesis}]\label{jan15a}
  Suppose $F \in L^\infty_{r \times r}$ has invertible values a.e.
  Then $F$ has the form $F=H^*G$ a.e.\ for some $G,H$ in $H^\infty_{r
    \times r}$ if and only if $\log |\det F|$ is integrable.  In this
  case, for every $\varepsilon > 0$ such a factorization can be found
  with
  \begin{equation*}
    \| G \|_\infty \| H \|_\infty < \| F \|_\infty + \varepsilon .
  \end{equation*}
\end{theorem}

The proof of Theorem \ref{jan15a} in \cite{BarclayThesis} is long and
technical.  In fact, Barclay proves an $L^p$-version of this result
for all $p$, $1 \le p \le \infty$.

Another type of generalization is factorization with indices.  We
quote one result to illustrate this notion.

\begin{theorem}
  Let $F$ be an $r \times r$ matrix-valued function with rational
  entries.  Assume that $F$ has no poles on $\thecircle$ and that
  $\det F(\zeta) \neq 0$ for all $\zeta$ in $\thecircle$.  Then there
  exist integers $\varkappa_1 \le \varkappa_2 \le \dots \le
  \varkappa_r$ such that
  \begin{equation*}
    F(z) = F_-(z) \diag \{z^{\varkappa_1},\dots,z^{\varkappa_r} \}
                        F_+(z),
  \end{equation*}
  where $F_\pm$ are $r \times r$ matrix-valued functions with rational
  entries such that
  \begin{enumerate}
  \item[(i)] $F_+(z)$ has no poles for $|z| \le 1$ and $\det F_+(z)
    \neq 0$ for $|z| \le 1$;
  \item[(ii)] $F_-(z)$ has no poles for $|z| \ge 1$ including
    $z=\infty$ and $\det F_-(z) \neq 0$ for $|z| \ge 1$ including $z =
    \infty$.
  \end{enumerate}
\end{theorem}

The case in which $F$ is nonnegative on $\thecircle$ can be handled
using the operator \fr theorem (the indices are all zero in this
case).  The general case is given in Gohberg, Goldberg, and Kaashoek
\cite[pp.\ 236--239]{GGKvolI}.  This is a large subject that includes,
for example, general theories of factorization in Bart, Gohberg,
Kaashoek, and Ran \cite{BGKR} and Clancey and Gohberg
\cite{ClanceyGohbergOT3}.

\medskip
\noindent {\bf Historical remarks:}
\smallskip

Historical accounts of spectral factorization appear in \cite{BC1992,
  Helson1964, RRbook, Rozanov1967, NFbook}.  Briefly, the problem of
factoring nonnegative matrix-valued functions on the unit circle rose
to prominence in the prediction theory of multivariate stationary
stochastic processes.  The first results of this theory were announced
by Zasuhin \cite{zasuhin} without complete proofs; proofs were
supplied by M.~G.\ Kre{\u\i}n in lectures.  Modern accounts of
prediction theory and matrix generalizations of Szeg{\H o}'s theorem
are based on fundamental papers of Helson and Lowdenslager
\cite{HelsonLowdenslagerI, HelsonLowdenslagerII}, and Wiener and
Masani \cite{WienerMasaniI, WienerMasaniII}.  The general case of
Theorem \ref{nov23c} is due to Devinatz \cite{Devinatz1961}; other
proofs are given in \cite{Douglas1966, Helson1964, RRbook}.  For an
engineering view and computational methods, see \cite[Chapter
8]{KSH2000} and \cite{SayedKailath2001}.

The original source for Lowdenslager's Criterion (Lemmas \ref{sep22a}
and \ref{nov15a}) is \cite{lowdenslager1963}; an error in
\cite{lowdenslager1963} was corrected by Douglas \cite{Douglas1966}.
There is a generalization, given by Sz.-Nagy and Foias \cite[pp.\ 
201--203]{NFbook}, in which the isometry may have a nontrivial unitary
component and the shift component yields a maximal factorable summand.
Lowdenslager's Criterion is used in the construction of canonical
models of operators by de~Branges \cite{deBranges1968}.  See also
Constantinescu \cite{constantinescu1990} for an adaptation to Toeplitz
kernels and additional references.

\goodbreak
\section{Multivariable theory}\label{S:multivar}

It is natural to wonder to what extent the results for one variable
carry over to several variables.  Various interpretations of ``several
variables'' are possible.  The most straightforward is to consider
Laurent polynomials in complex variables $z_1,\ldots, z_d$ that are
nonnegative on the $d$-torus $\thecircle^d$.  The method of Schur
complements in \S\ref{S:Schur} suggests an approach to the
factorization problem for such polynomials.  Care is needed, however,
since the first conjectures for a multivariable \fr theorem that might
come to mind are false, as explained below.  Multivariable
generalizations of the \fr theorem are thus necessarily weaker than
the one-variable result.  One difficulty has to do with degrees, and
if the condition on degrees is relaxed, there is a neat result in the
strictly positive case (Theorem \ref{jan16a}).

By a Laurent polynomial in $z = (z_1,\dots,z_d)$ we understand an
expression
\begin{equation}
  \label{jan16b}
  Q(z) = \sum_{k_1=-m_1}^{m_1} \cdots \sum_{k_d=-m_d}^{m_d}
  Q_{k_1,\dots,k_d} z_1^{k_1} \cdots z_d^{k_d} .
\end{equation}
We assume that the coefficients belong to $\lg$, where $\eG$ is a
Hilbert space.  With obvious interpretations, the scalar case is
included.  By an analytic polynomial with coefficients in $\lg$ we
mean an analogous expression, of the form
\begin{equation}
  \label{jan16c}
  P(z) = \sum_{k_1=0}^{m_1} \cdots \sum_{k_d=0}^{m_d}
  P_{k_1,\dots,k_d} z_1^{k_1} \cdots z_d^{k_d} .
\end{equation}
The numbers $m_1,\dots,m_d$ in \eqref{jan16b} and \eqref{jan16c} are
upper bounds for the degrees of the polynomials in $z_1,\dots,z_d$,
which we define as the smallest values of $m_1,\dots,m_d$ that can be
used in the representations \eqref{jan16b} and \eqref{jan16c}.

Suppose that $Q(z)$ has the form \eqref{jan16b} and satisfies
$Q(\zeta) \ge 0$ for all $\zeta \in \thecircle^d$, that is, for all
$\zeta = (\zeta_1,\dots,\zeta_d)$ with $|\zeta_1| = \cdots =
|\zeta_d|=1$.  Already in the scalar case, one cannot always find an
analytic polynomial $P(z)$ such that $Q(\zeta) = P(\zeta)^*P(\zeta)$,
$\zeta \in \thecircle^d$.  This was first explicitly shown by Lebow
and Schreiber \cite{MR684391}.  There are also difficulties in writing
$Q(\zeta) = \sum _{j=1}^r P_j(\zeta)^*P_j(\zeta)$, $\zeta \in
\thecircle^d$, for some finite set of analytic polynomials, at least
if one requires that the degrees of the analytic polynomials do not
exceed those of $Q(z)$ as in the one-variable case (see Naftalevich
and Schreiber \cite{MR803356}, Rudin \cite{MR0151796}, and Sakhnovich
\cite[\S 3.6]{LAS1997}).  The example in \cite{MR803356} is based on a
Cayley transform of a version of a real polynomial over $\mathbb R^2$
called Motzkin's polynomial, which was the first explicit example of a
nonnegative polynomial in $\mathbb R^d$, $d>1$, which is not a sum of
squares of polynomials.  What is not mentioned in these sources is
that if we loosen the restriction on degrees, the polynomial in
\cite{MR803356} can be written as a sum of squares
(see~\cite{DW2005}).  Nevertheless, for three or more variables, very
general results of Scheiderer \cite{Scheiderer1} imply that there
exist nonnegative, but not strictly positive, polynomials which cannot
be expressed as such finite sums regardless of degrees.

\begin{theorem}\label{jan16a}
  Let $Q(z)$ be a Laurent polynomial in $z = (z_1,\dots,z_d)$ with
  coefficients in $\lg$ for some Hilbert space $\eG$.  Suppose that
  there is a $\delta > 0$ such that $Q(\zeta) \ge \delta I$ for all $\zeta
  \in\thecircle^d$.  Then
\begin{equation}\label{jan16ee}
  Q(\zeta) = \sum_{j=1}^r P_j(\zeta)^*P_j(\zeta), 
           \qquad \zeta \in \thecircle^d,
\end{equation}
for some analytic polynomials $P_1(z),\dots,P_r(z)$ in $z =
(z_1,\dots,z_d)$ which have coefficients in $\lg$.  Furthermore, for
any fixed $k$, the representation $\eqref{jan16ee}$ can be chosen such
that the degree of each analytic polynomial in $z_k$ is no more than
the degree of $Q(z)$ in $z_k$.
\end{theorem}

The scalar case of Theorem \ref{jan16a} follows by a theorem of
Schm\"udgen \cite{Schmudgen}, which states that strictly positive
polynomials over compact semialgebraic sets in $\mathbb R^n$ (that is,
sets which are expressible in terms of finitely many polynomial
inequalities) can be written as weighted sums of squares, where the
weights are the polynomials used to define the semialgebraic set (see
also \cite{MR759099}); the proof is nonconstructive.  On the other
hand, the proof we sketch using Schur complements covers the
operator-valued case, and it gives an algorithm for finding the
solution.  One can also give estimates for the degrees of the
polynomials involved, though we have not stated these.

We prove Theorem \ref{jan16a} for the case $d=2$, following Dritschel
\cite{MAD2004}.  The general case is similar.  The argument mimics the
method of Schur complements, especially in its original form used in
\cite{MAD2004}.  In place of Toeplitz matrices whose entries are
operators, in the case of two variables we use Toeplitz matrices whose
entries are themselves Toeplitz matrices.  The fact that the first
level Toeplitz blocks are infinite in size causes problems, and so we
truncate these blocks to finite size.  Then everything goes through,
but instead of factoring the original polynomial $Q(z)$, the result is
a factorization of polynomials $Q^{(N)}(z)$ that are close to $Q(z)$.
When $Q(\zeta) \ge \delta I$ on $\thecircle^d$ for some $\delta > 0$,
there is enough wiggle room to factor $Q(z)$ itself.  We isolate the
main steps in a lemma.

\begin{lemma}\label{jan17a}
  Let 
  \begin{equation*}
     Q(z) = \sum_{j=-m_1}^{m_1} \sum_{k=-m_2}^{m_2}\,
  Q_{jk} z_1^j z_2^k 
  \end{equation*}
  be a Laurent polynomial with coefficients in $\lg$ such that
  $Q(\zeta) \ge 0$ for all $\zeta = (\zeta_1,\zeta_2)$ in
  $\thecircle^2$.  Set
  \begin{equation*}
         Q^{(N)}(z) = \sum_{j=-m_1}^{m_1} \sum_{k=-m_2}^{m_2}
           \frac{N+1-|k|}{N+1} \,
            Q_{jk}\, z_1^j z_2^k .
  \end{equation*}
  Then for each $N \ge m_2$, there are analytic polynomials
  \begin{equation}\label{jan18d}
    F_\ell(z) = \sum_{j=0}^{m_1} \sum_{k=0}^{N}
  F_{jk}^{(\ell)} \, z_1^j z_2^k ,
  \qquad \ell = 0,\dots,N,
  \end{equation}
  with coefficients in $\lg$ such that
  \begin{equation}\label{jan18c}
    Q^{(N)}(\zeta) = \sum_{\ell=0}^N F_\ell(\zeta)^*F_\ell(\zeta),
    \qquad \zeta \in\thecircle^2 .
  \end{equation}
\end{lemma}

\begin{proof}
  Write
  \begin{equation*}
     Q(z) = \sum_{j=-m_1}^{m_1} \bigg(\sum_{k=-m_2}^{m_2}\,
  Q_{jk} z_2^k \bigg)  z_1^j
  =
      \sum_{j=-m_1}^{m_1} R_j(z_2)\, z_1^j ,
  \end{equation*}
  and extend all sums to run from $-\infty$ to $\infty$ by setting
  $Q_{jk} = 0$ and $R_j(z_2)=0$ if $|j|>m_1$ or $|k|>m_2$.  Introduce
  a Toeplitz matrix $T$ whose entries are the Toeplitz matrices $T_j$
  corresponding to the Laurent polynomials $R_j(z_2)$, that is,
  \begin{equation*}
    T = \begin{pmatrix}
    T_{0}   & T_{-1}   & T_{-2} & \cdots\\
    T_{1}   & T_{0}    & T_{-1} & \ddots \\
    T_{2}   & T_{1}    & T_{0} & \ddots \\
    \vdots    & \ddots  & \ddots & \ddots
  \end{pmatrix} ,
\qquad
T_j
=
    \begin{pmatrix}
    Q_{j0}   & Q_{j,-1}   & Q_{j,-2} & \cdots\\
    Q_{j1}   & Q_{j0}    & Q_{j,-1} & \ddots \\
    Q_{j2}   & Q_{j1}    & Q_{j0} & \ddots \\
    \vdots    & \ddots  & \ddots & \ddots
  \end{pmatrix},
  \end{equation*}
  $j = 0,\pm 1, \pm 2, \dots \,$.  Notice that $T$ is finitely banded,
  since $T_j = 0$ for $|j| > m_1$.  The identity \eqref{aug28a} has
  the following generalization:
  \begin{equation*}
    \ip{Th}{h} = \sum_{p=0}^\infty \sum_{q=0}^\infty 
                           \ip{T_{q-p}h_p}{h_q}  
=
  \int_{\thecircle^2} {\ip{Q(\zeta)h(\zeta)}{h(\zeta)}}_\eG\; 
                               d\sigma_2(\zeta) \,.
  \end{equation*}
  Here $\zeta = (\zeta_1,\zeta_2)$ and $d\sigma_2(\zeta) =
  d\sigma(\zeta_1) d\sigma(\zeta_2)$.  Also,
  \begin{equation*}
    h(\zeta) = \sum_{p=0}^\infty \sum_{q=0}^\infty
                  h_{pq} \, \zeta_1^p \zeta_2^q \,,
  \end{equation*}
  where the coefficients are vectors in $\eG$ and all but finitely
  many are zero, and
  \begin{equation*}
    h = 
    \begin{pmatrix}
      h_{0} \\ h_{1} \\  \vdots
    \end{pmatrix},
\qquad
h_p = 
    \begin{pmatrix}
      h_{p0} \\ h_{p1} \\  \vdots
    \end{pmatrix},
\qquad p=0,1,2,\dots \,.
  \end{equation*}
  It follows that $T$ acts as a bounded operator on a suitable direct
  sum of copies of~$\eG$.  Since $Q(\zeta) \ge 0$ on $\thecircle^2$,
  $T \ge 0$.

  Fix $N \ge m_2$.  Set
  \begin{equation*}
    T' = \begin{pmatrix}
    T_{0}'   &T_{-1}'   &T_{-2}' & \cdots\\
    T_{1}'   &T_{0}'    &T_{-1}' & \ddots \\
    T_{2}'   &T_{1}'    &T_{0}' & \ddots \\
    \vdots    & \ddots  & \ddots & \ddots
  \end{pmatrix}  ,
  \end{equation*}
  where $T_j'$ is the upper $(N+1)\times(N+1)$ block of $T_j$ with a
  normalizing factor:
  \begin{equation*}
    T_j' = \frac{1}{N+1}
    \begin{pmatrix}
    Q_{j0}   & Q_{j,-1}  &\cdots   & Q_{j,-N} \\
    Q_{j1}   & Q_{j0}    &\cdots   & Q_{j,-N+1} \\
         && \cdots \\
    Q_{jN}   & Q_{j,N-1}    &\cdots   & Q_{j0} \\
   \end{pmatrix},
   \qquad j = 0,\pm 1, \pm 2, \dots \,.
  \end{equation*}
  Then $T'$ is the Toeplitz matrix corresponding to the Laurent
  polynomial
  \begin{equation*}
    \Psi(w) = \sum_{j=-m_1}^{m_1} T_j'\, w^j .
  \end{equation*}
  Moreover, $T' \ge 0$ since it is a positive constant multiple of a
  compression of $T$.  Thus $\Psi(w) \ge 0$ for $|w|=1$.  By the
  operator \fr theorem (Theorem \ref{ofr}),
  \begin{equation}\label{jan18a}
    \Psi(w) = \Phi(w)^*\Phi(w), \qquad |w|=1,
  \end{equation}
  for some analytic polynomial $\Phi(w) = \sum_{j=0}^{m_1} \Phi_j w^j$
  whose coefficients are $(N+1)\times(N+1)$ matrices with entries in
  $\lg$.  Write
  \begin{equation*}
    \Phi_j = 
    \begin{pmatrix}
      \Phi_{jN} & \Phi_{j,N-1} & \cdots & \Phi_{j0}
    \end{pmatrix} ,
  \end{equation*}
  where $\Phi_{jk}$ is the $k$-th column in $\Phi_j$.  Set
\begin{equation*}
  \widetilde F(z) = \sum_{j=0}^{m_1} \sum_{k=0}^{N}
        \Phi_{jk} \, z_1^j z_2^k \,.
\end{equation*}
The identity \eqref{jan18a} is equivalent to $2m_1 +1$ relations for
the coefficients of $\Psi(w)$.  The coefficients of $\Psi(w)$ are
constant on diagonals, there being $N+1-k$ terms in the $k$-th
diagonal above the main diagonal, and similarly below.  If these terms
are summed, the result gives $2m_1 + 1$ relations equivalent to the
identity
\begin{equation}\label{jan18b}
  Q^{(N)}(\zeta) = \widetilde F(\zeta)^*\widetilde F(\zeta), \qquad
  \zeta\in\thecircle^2. 
\end{equation}
We omit the calculation, which is straightforward but laborious.  To
convert \eqref{jan18b} to the form \eqref{jan18c}, write
\begin{equation*}
  \Phi_{jk}
=
  \begin{pmatrix}
      F^{(0)}_{jk} \\ F^{(1)}_{jk} \\ \vdots \\ F^{(N)}_{jk}
    \end{pmatrix} ,
\qquad j = 0,\dots,m_1 \text{\;\;and\;\;} k=0,\dots,N.
\end{equation*}
Then 
\begin{equation*}
  \widetilde F(z) 
=
\begin{pmatrix}
  F_0(z) \\ F_1(z) \\ \vdots \\ F_N(z)
\end{pmatrix} ,
\end{equation*}
where $F_0(z),\dots,F_N(z)$ are given by \eqref{jan18d}, and so
\eqref{jan18b} takes the form \eqref{jan18c}.
\end{proof}

\begin{proof}[Proof of Theorem $\ref{jan16a}$ for the case $d=2$]
  Suppose $N \ge m_2$, and set
  \begin{equation*}
    \widetilde Q(z)
=
\sum_{j=-m_1}^{m_1} \sum_{k=-m_2}^{m_2}
           \frac{N+1}{N+1-|k|} \,
            Q_{jk}\, z_1^j z_2^k .
  \end{equation*}
  The values of $\widetilde Q(z)$ are selfadjoint on $\thecircle^2$,
  and $\widetilde Q(z) = Q(z) + S(z)$, where
\begin{equation*}
  S(z) = \sum_{j=-m_1}^{m_1} \sum_{k=-m_2}^{m_2}
           \frac{|k|}{N+1-|k|} \,
            Q_{jk}\, z_1^j z_2^k .
\end{equation*}
Now choose $N$ large enough that $\| S(\zeta) \| < \delta$, $\zeta \in
\thecircle^2$.  Then $\widetilde Q(\zeta) \ge 0$ on $\thecircle^2$,
and the result follows on applying Lemma \ref{jan17a} to $\widetilde
Q(z)$.
\end{proof}

Further details can be found in \cite{MAD2004}, and a variation on
this method yielding good numerical results is given in Geronimo and
Lai \cite{GL2006}.

While, as we mentioned, there is in general little hope of finding a
factorization of a positive trigonometric polynomial in two or more
variables in terms of one or more analytic polynomials of the same
degree, it happens that there are situations where the existence of
such a factorization is important.  In particular, Geronimo and
Woerdeman consider this question in the context of the autoregressive
filter problem \cite{gw1,gw2}, with the first paper addressing the
scalar case and the second the operator-valued case, both in two
variables.  They show that for scalar-valued polynomials in this
setting there exists a factorization in terms of a single
\textbf{stable} (so invertible in the bidisk $\mathbb D^2$) analytic
polynomial of the same degree if and only if a full rank condition
holds for certain submatrices of the associated Toeplitz matrix
(\cite[Theorem~1.1.3]{gw1}).  The condition for operator-valued
polynomials is similar, but more complicated to state.  We refer the
reader to the original papers for details.

Stable scalar polynomials in one variable are by definition outer, so
the Geronimo and Woerdeman results can be viewed as a statement about
outer factorizations in two variables.  In \cite{DW2005}, a different
notion of outerness is considered.  As we saw in \S\ref{S:Schur}, in
one variable outer factorizations can be extracted using Schur
complements.  The same Schur complement method in two or more
variables gives rise to a version of ``outer'' factorization which in
general does not agree with that coming from stable polynomials.  In
\cite{DW2005}, this Schur complement version of outerness is used when
considering outer factorizations for polynomials in two or more
variables.  As in the Geronimo and Woerdeman papers, it is required
that the factorization be in terms of a single analytic polynomial of
the same degree as the polynomial being factored.  Then necessary and
sufficient conditions for such an outer factorization under these
constraints are found (\cite[Theorem~4.1]{DW2005}).

The problem of spectral factorization can also be considered in the
multivariable setting.  Blower \cite{blower} has several results along
these lines for bivariate matrix-valued functions, including a matrix
analogue of Szeg\H{o}'s theorem similar to Theorem~\ref{nov23cc}.  His
results are based on a two-variable matrix version of
Theorem~\ref{jan16a}, and the arguments he gives coupled with
Theorem~\ref{jan16a} can be used to extend these results to
polynomials in $d>2$ variables as well.

\goodbreak
\section{Noncommutative factorization}\label{S:nc_fact}

We now present some noncommutative interpretations of the notion of
``several variables,'' starting with the one most frequently
considered, and for which there is an analogue of the \fr theorem.  It
is due to Scott McCullough and comes very close to the one-variable
result.  Further generalizations have been obtained by Helton,
McCullough and Putinar in~\cite{HMP3}.  For a broad overview of the
area, two nice survey articles have recently appeared by Helton and
Putinar \cite{HeltonPutinar} and Schm\"udgen \cite{schmuedgen}
covering noncommutative real algebraic geometry, of which the
noncommutative analogues of the \fr theorem are one aspect.

In keeping with the assumptions made in \cite{MR1815959}, all Hilbert
spaces in this section are taken to be separable.  Fix Hilbert spaces
$\eG$ and $\eH$, and assume that $\eH$ is infinite dimensional.

Let $S$ be the free semigroup with generators $a_1,\dots,a_d$.  Thus
$S$ is the set of words
\begin{equation}\label{jan24a}
  w = a_{j_1}\cdots a_{j_k},
\quad
j_1,\dots,j_k \in \{1,\dots,d\},
\quad
k=0,1,2,\dots \,,
\end{equation}
with the binary operation concatenation.  The empty word is denoted
$e$.  The length of the word \eqref{jan24a} is $|w|=k$ (so $|e|=0$).
Let $S_m$ be the set of all words \eqref{jan24a} of length at
most~$m$.  The cardinality of $S_m$ is $\ell_m = 1 + d + d^2 + \cdots
+ d^m$.

We extend $S$ to a free group $G$.  We can think of the elements of
$G$ as words in $a_1,\dots,a_d,a_1^{-1},\dots,a_d^{-1}$, with two such
words identified if one can be obtained from the other by cancelling
adjacent terms of the form $a_j$ and $a_j^{-1}$.  The binary operation
in $G$ is also concatenation.  Words in $G$ of the form $h = v^{-1}w$
with $v,w \in S$ play a special role and are called
\textbf{hereditary}.  Notice that a hereditary word $h$ has many
representations $h = v^{-1}w$ with $v,w \in S$.  Let $H_m$ be the set
of hereditary words $h$ which have at least one representation in the
form $h = v^{-1}w$ with $v,w \in S_m$.

We can now introduce the noncommutative analogues of Laurent and
analytic polynomials.  A hereditary polynomial is a formal expression
\begin{equation}\label{jan24b}
  Q = \sum_{h \in H_m}  h \otimes Q_h ,
\end{equation}
where $Q_h\in\lg$ for all $h$.  Analytic polynomials are hereditary
polynomials of the special form
\begin{equation}\label{jan24c}
  P = \sum_{w \in S_m}  w \otimes P_w ,
\end{equation}
where $P_w\in\lg$ for all $w$.  The identity
\begin{equation*}
  Q=P^*P
\end{equation*}
is defined to mean that
\begin{equation*}
    Q_h = \sum_{\genfrac{}{}{0pt}{1}{v,w \in S_m}{h=v^{-1}w}}
                   P_v^* P_w ,
   \qquad h \in H_d.
\end{equation*}
 
Next we give meaning to the expressions $Q(U)$ and $P(U)$ for
hereditary and analytic polynomials \eqref{jan24b} and \eqref{jan24c}
and any tuple $U = (U_1,\dots,U_d)$ of unitary operators on $\eH$.
First define $U^w \in\lh$ for any $w \in S$ by writing $w$ in the form
\eqref{jan24a} and setting
\begin{equation*}
  U^w = U_{j_1}\cdots U_{j_k} .
\end{equation*}
By convention, $U^e = I$ is the identity operator on~$\eH$.  If $h \in
\eG$ is a hereditary word, set
\begin{equation*}
  U^h = (U^v)^* U^w
\end{equation*}
for any representation $h = v^{-1}w$ with $v,w \in S$; this definition
does not depend on the choice of representation.  Finally, define
$Q(U), P(U) \in \eL(\eH \otimes \eG)$ by
\begin{equation*}
  Q(U) = \sum_{h \in H_m}  U^h \otimes Q_h \,,
\qquad
  P(U) =\sum_{w \in S_m}  U^w \otimes P_w.
\end{equation*}
The reader is referred to, for example, Murphy \cite[\S6.3]{Murphy}
for the construction of tensor products of Hilbert spaces and
algebras, or Palmer, \cite[\S1.10]{Palmer1} for a more detailed
account.

\goodbreak
\begin{theorem}[McCullough \cite{MR1815959}]\label{jan24d}
  Let
  \begin{equation*}
    Q = \sum_{h \in H_m} h \otimes Q_h
  \end{equation*}
  be a hereditary polynomial with coefficients in $\lg$ such that
  $Q(U)) \ge 0$ for every tuple $U = (U_1,\dots,U_d)$ of unitary
  operators on~$\eH$.  Then for some $\ell \le \ell_m$, there exist
  analytic polynomials
\begin{equation*}
  P_j = \sum_{w \in S_m}  w \otimes P_{j,w} ,
\qquad
  j=1,\dots,\ell,
\end{equation*}
with coefficients in $\lg$ such that
 \begin{equation*}
   Q = P_1^*P_1 + \cdots + P_{\ell}^* P_{\ell}.
 \end{equation*}
 Moreover, for any tuple $U = (U_1,\dots,U_d)$ of unitary operators on
 $\eH$,
 \begin{equation*}
   Q(U) = P_1(U)^*P_1(U) + \cdots + P_{\ell}(U)^* P_{\ell}(U).
 \end{equation*}
 In these statements, when $\eG$ is infinite dimensional, we can
 choose~$\ell=1$.
\end{theorem}

As noted by McCullough, when $d=1$, Theorem \ref{jan24d} gives a
weaker version of Theorem 2.1.  However, Theorem 2.1 can be deduced
from this by a judicious use of Beurling's theorem and an inner-outer
factorization.

McCullough's theorem uses one of many possible choices of
noncommutative spaces on which some form of trigonometric polynomials
can be defined.  We place this, along with the commutative versions,
within a general framework, which we now explain.

The complex scalar-valued trigonometric polynomials in $d$ variables
form a unital $*$-algebra $\eP$, the involution taking $z^n$ to
$z^{-n}$, where for $n = (n_1,\ldots, n_d)$, $-n = (-n_1,\ldots,
-n_d)$.  If instead the coefficients are in the algebra $\eC = \lg$
for some Hilbert space $\eG$, then the unital involutive algebra of
trigonometric polynomials with coefficients in $\eC$ is
$\eP\otimes\eC$.  The unit is $1\otimes 1$.  A representation of
$\eP\otimes\eC$ is a unital algebra $*$-homomorphism from
$\eP\otimes\eC$ into $\lh$ for a Hilbert space $\eH$.  The key thing
here is that $z_1, \ldots , z_d$ generate $\eC$, and so assuming we do
not mess with the coefficient space, a representation $\pi$ is
determined by specifying $\pi(z_k)$, $k=1,\ldots, d$.

First note that since $z_k^*z_k = 1$, $\pi(z_k)$ is isometric, and
since $z_k^* = z_k^{-1}$, we then have that $\pi(z_k)$ is unitary.
Assuming the variables commute, the $z_k$s generate a commutative
group $G$ which we can identify with $\mathbb Z^d$ under addition, and
the irreducible representations of commutative groups are one
dimensional.  This essentially follows from the spectral theory for
normal operators (see, for example, Edwards \cite[p.\ 718]{Edwards}).
However, the one-dimensional representations are point evaluations on
$\mathbb T^d$.  Discrete groups with the discrete topology are
examples of locally compact groups.  Group representations of locally
compact groups extend naturally to the algebraic group algebra, which
in this case is $\eP$, and then on to the algebra $\eP\otimes\eC$ by
tensoring with the identity representation of $\eC$.  So a seemingly
more complex way of stating that a commutative trigonometric
polynomial $P$ in several variables is positive / strictly positive is
to say that for each (topologically) irreducible unitary
representation $\pi$ of $G$, the extension of $\pi$ to a unital
$*$-representation of the algebra $\eP\otimes\eC$, also called $\pi$,
has the property that $\pi(P) \geq 0$ / $\pi(P) > 0$.  By the way,
since $\mathbb T^d$ is compact, $\pi(P) > 0$ implies the existence of
some $\epsilon > 0$ such that $\pi(P - \epsilon 1\otimes 1) = \pi(P) -
\epsilon 1 \geq 0$.

What is gained through this perspective is that we may now define
noncommutative trigonometric polynomials over a finitely generated
discrete (so locally compact) group $G$ in precisely the same manner.
These are the elements of the algebraic group algebra $\eP$ generated
by $G$; that is, formal complex linear combinations of elements of $G$
endowed with pointwise addition and a convolution product (see Palmer
\cite[section~1.9]{Palmer1}).  Then a trigonometric polynomial in
$\eP\otimes\eC$ is formally a finite sum over $G$ of the form $P =
\sum_g g\otimes P_g$ where $P_g\in \eC$ for all $g$.

We also introduce an involution by setting $g^* = g^{-1}$ for $g\in
G$.  A trigonometric polynomial $P$ is \textbf{self adjoint} if for
all $g$, $P_{g^*} = P_g^*$.  There is an order structure on
selfadjoint elements defined by saying that a selfadjoint polynomial
$P$ is \textbf{positive / strictly positive} if for every irreducible
unital $*$-representation $\pi$ of $G$, the extension as above of
$\pi$ to the algebra $\eP\otimes\eC$ (again called $\pi$), satisfies
$\pi(P) \geq 0$ / $\pi(P) > 0$; where by $\pi(P) > 0$ we mean that
there exists some $\epsilon > 0$ independent of $\pi$ such that $\pi(P
- \epsilon (1\otimes 1)) \geq 0$.  Letting $\Omega$ represent the set
of such irreducible representations, we can in a manner suggestive of
the Gel'fand transform define $\hat{P} (\pi) = \pi(P)$, and in this
way think of $\Omega$ as a sort of noncommutative space on which our
polynomial is defined.  The Gel'fand-Ra{\u\i}kov theorem (see, for
example, Palmer \cite[Theorem 12.4.6]{Palmer2}) ensures the existence
of sufficiently many irreducible representations to separate $G$, so
in particular, $\Omega \neq\emptyset$.

For a finitely generated discrete group $G$ with generators $\{a_1,
\ldots,a_d\}$, let $S$ be a fixed unital subsemigroup of $G$
containing the generators.  The most interesting case is when $S$ is
the subsemigroup generated by $e$ (the group identity) and $\{a_1,
\ldots,a_d\}$.  As an example of this, if $G$ is the noncommutative
free group in $d$ generators, then the unital subsemigroup generated
by $\{a_1, \ldots,a_d\}$ consists of group elements $w$ of the form
$e$ (for the empty word) and those which are an arbitrary finite
product of positive powers of the generators, as in \eqref{jan24a}.

We also need to address the issue of what should play the role of
Laurent and analytic trigonometric polynomials in the noncommutative
setting.  The \textbf{hereditary} trigonometric polynomials are
defined as those polynomials of the form $P = \sum_j w_{j1}^*
w_{j2}\otimes P_j$, where $w_{j1}, w_{j2} \in S$.  We think of these
as the Laurent polynomials.  Trigonometric polynomials over $S$ are
referred to as \textbf{analytic} polynomials.  The \textbf{square} of
an analytic polynomial $Q$ is the hereditary trigonometric polynomial
$Q^*Q$.  Squares are easily seen to be positive.  As a weak analogue
of the Fej\'er-Riesz theorem, we prove a partial converse below.

We refer to those hereditary polynomials which are selfadjoint as real
hereditary polynomials, and denote the set of such polynomials by $H$.
While these polynomials do not form an algebra, they are clearly a
vector space.  Those which are finite sums of squares form a cone $C$
in $H$ (that is, $C$ is closed under sums and positive scalar
multiplication).  Any real polynomial is the sum of terms of the form
$1\otimes A$ or $w_2^*w_1\otimes B + w_1^*w_2\otimes B^*$, where
$w_1, w_2\in S$ and $A$ is selfadjoint.  The first of these is
obviously the difference of squares.  Using $w^*w = 1$ for any $w\in
G$, we also have
\begin{equation*}
  w_2^*w_1\otimes B +  w_1^*w_2 \otimes B^* = (w_1\otimes B + 
  w_2\otimes 1)^*(w_1\otimes B + w_2\otimes 1) - 1\otimes (1+B^*B).
\end{equation*}
Hence $H = C - C$.

For $A, B\in \lh$ and $w_1,w_2\in S$,
\begin{equation*}
  \begin{split}
    0 & \leq (w_1\otimes A + w_2\otimes B)^*(w_1\otimes A + 
    w_2\otimes B) \\
    & \leq (w_1\otimes A + w_2\otimes B)^*(w_1\otimes A + 
    w_2\otimes B) \\
   &\hskip1cm  + (w_1\otimes A - w_2\otimes B)^*(w_1\otimes A -
    w_2\otimes B) \\
    &= 2(1\otimes A^*A + 1\otimes B^*B ) \\
    & \leq (\|A\|^2 + \|B\|^2) (1\otimes 1).
  \end{split}
\end{equation*}
Applying this iteratively, we see that for any $P\in H$, there is some
constant $0\leq \alpha < \infty$ such that $\alpha 1 \pm P \in C$.  In
other words, the cone $C$ is \textbf{archimedean}.  In particular,
$1\otimes 1$ is in the algebraic interior of $C$, meaning that if
$P\in H$, then there is some $0 < t_0 \leq 1$ such that for all
$0<t<t_0$, $t(1\otimes 1) + (1-t)P \in C$.

\begin{theorem}\label{jan12a}
  Let $G$ be a finitely generated discrete group, $P$ a strictly
  positive trigonometric polynomial over $G$ with coefficients in
  $\lg$.  Then $P$ is a sum of squares of analytic polynomials.
\end{theorem}

\begin{proof}
  The proof uses a standard GNS construction and separation argument.
  Suppose that for some $\epsilon > 0$, $P - \epsilon (1\otimes 1)
  \geq 0$ but that $P\notin C$.  Since $C$ has nonempty algebraic
  interior, it follows from the Edelheit-Kakutani theorem\footnote{
    \cite{Holmes} Holmes, Corollary, \S4B.  \textsl{Let $A$ and $B$ be
      nonempty convex subsets of $X$, and assume the algebraic
      interior of $A$, $\mathrm{cor}(A)$ is nonempty.  Then $A$ and
      $B$ can be separated if and only if $\mathrm{cor}(A)\cap B =
      \emptyset$.}} that there is a nonconstant linear functional
  $\lambda: H \to \mathbb R$ such that $\lambda(C) \geq 0$ and
  $\lambda(P) \leq 0$.  Since $\lambda$ is nonzero, there is some
  $R\in H$ with $\lambda(R) > 0$, and so since the cone $C$ is
  archimedean, there exists $\alpha > 0$ such that $\alpha (1\otimes
  1) - R \in C$.  From this we see that $\lambda (1 \otimes 1) > 0$,
  and so by scaling, we may assume $\lambda(1\otimes 1) = 1$.

  We next define a nontrivial scalar product on $H$ by setting
  \begin{equation*}
    \ip{w_1\otimes A}{w_2\otimes B} = \lambda(w_2^*w_1\otimes B^*A)
  \end{equation*}
  and extending linearly to all of $H$.  It is easily checked that
  this satisfies all of the properties of an inner product, except
  that $\ip{w\otimes A}{w\otimes A} = 0$ may not necessarily imply
  that $w\otimes A = 0$.  Even so, such scalar products satisfy the
  Cauchy-Schwarz inequality, and so $N = \{w\otimes A : \ip{w\otimes A
  }{w\otimes A} = 0\}$ is a vector subspace of $H$.  Therefore this
  scalar product induces an inner product on $H/N$, and the completion
  $\eH$ of $H/N$ with respect to the associated norm makes $H/N$ into
  a Hilbert space.

  We next define a representation $\pi: H \to \lh$ by the left regular
  representation; that is, $\pi(P)[w\otimes A] = [P(w\otimes A)]$,
  where $[\,\cdot\,]$ indicates an equivalence class in $H/N$.  Since
  $P\geq \epsilon (1\otimes 1) \geq 0$ for some $\epsilon > 0$,
  $P-\epsilon/2 (1\otimes 1) > 0$.  Suppose that $P\notin C$.  Then
  \begin{equation*}
    \lambda((P-\epsilon/2 (1\otimes 1)) + \epsilon/2 (1\otimes 1)) =
    \lambda(P-\epsilon/2 (1\otimes 1)) + \epsilon/2 \leq 0.
  \end{equation*}
  Hence
  \begin{equation*}
    \ip{\pi(P-\epsilon/2 (1\otimes 1))[1\otimes 1]}{[1\otimes 1]} \leq
    -\epsilon/2,
  \end{equation*}
  and so $\pi(P-\epsilon/2 (1\otimes 1)) \not\geq 0$.  The
  representation $\pi$ obviously induces a unitary representation of
  $G$ via $\pi(a_i) = \pi(a_i\otimes 1)$, where $a_i$ is a generator
  of $G$.  The (irreducible) representations of $G$ are in bijective
  correspondence with the essential unital $*$-representations of the
  group $C^*$-algebra $C^*(G)$ (Palmer,
  \cite[Theorem~12.4.1]{Palmer2}), which then restrict back to
  representations of $H$.  Since unitary representations of $G$ are
  direct integrals of irreducible unitary representations (see, for
  example, Palmer, \cite[p.~1386]{Palmer2}), there is an irreducible
  unitary representation $\pi'$ of $G$ such that the corresponding
  representation of $H$ has the property that $\pi'(P-\epsilon/2
  (1\otimes 1)) \not\geq 0$, giving a contradiction.
\end{proof}

The above could equally well have been derived using any $C^*$-algebra
in the place of $\lh$.  One could also further generalize to
non-discrete locally compact groups, replacing the trigonometric
polynomials by functions of compact support.

We obtain Theorem~\ref{jan16a} as a corollary if we take $G$ to be
the free group in $d$ commuting letters.  On the other hand, if $G$ is
the noncommutative free group on $d$ letters, it is again
straightforward to specify the irreducible representations of $G$.
These take the generators $(a_1,\ldots,a_d)$ to irreducible $d$-tuples
$(U_1,\ldots,U_d)$ of (noncommuting) unitary operators, yielding a
weak form of McCullough's theorem.

As mentioned earlier, it is known by results of Scheiderer
\cite{Scheiderer1} that when $G$ is the free group in $d$
\textit{commuting} letters, $d\geq 3$, there are positive polynomials
which cannot be expressed as sums of squares of analytic polynomials,
so no statement along the lines of Theorem~\ref{jan24d} can be true
for trigonometric polynomials if it is to hold for all finitely
generated discrete groups.  Just what can be said in various special
cases is still largely unexplored.

\goodbreak
\appendix
\section{Schur complements}\label{S:complement}

We prove the existence and uniqueness of Schur complements for Hilbert
space operators as required in Definition~\ref{nov4e}.

\begin{lemma}\label{sep27b}
  Let $T\in\lh$, where $\eH$ is a Hilbert space.  Let $\eK$ be a
  closed subspace of~$\eH$, and write
\begin{equation*}
  T = 
  \begin{pmatrix}
    A & B^* \\ B & C
  \end{pmatrix} \colon \eK \oplus \eK^\perp \to \eK \oplus \eK^\perp .
\end{equation*}
Then $T \ge 0$ if and only if $A \ge 0$, $C \ge 0$, and $B = C^\frac12
G A^\frac12$ for some contraction $G \in\eL(\eK,\eK^\perp)$.  The
operator $G$ can be chosen so that it maps $\ranclosure A$ into
$\ranclosure C$ and is zero on the orthogonal complement of
$\,\ranclosure A$, and then it is unique.
\end{lemma}

\begin{proof}
  If $B = C^\frac12 G A^\frac12$ where $G \in\eL(\eK,\eK^\perp)$ is a
  contraction, then
  \begin{equation*}
    T =   \begin{pmatrix}
  A^\frac12 & 0 \\
  C^\frac12 G & C^\frac12 (I-GG^*)^\frac12
  \end{pmatrix}
  \begin{pmatrix}
    A^\frac12 & G^* C^\frac12 \\
    0 & (I-GG^*)^\frac12 C^\frac12
  \end{pmatrix} \ge 0 .
  \end{equation*}
  Conversely, if $T \ge 0$, it is trivial that $A \ge 0$ and $C \ge
  0$.  Set
  \begin{equation*}
    N = T^\frac12 = 
    \begin{pmatrix}
      N_1 \\ N_2
    \end{pmatrix} \colon \eH \to \eH \oplus \eK .
  \end{equation*}
  Then $A = N_1N_1^*$ and $C=N_2N_2^*$, and so there exist partial
  isometries $V_1 \in \eL(\eK,\eH)$ and $V_2 \in \eL(\eK^\perp,\eH)$
  with initial spaces $\ranclosure A$ and $\ranclosure C$ such that
  $N_1^* = V_1 A^\frac12$ and $N_2^* = V_2 C^\frac12$.  Thus $B =
  N_2N_1^* = C^\frac12 G A^\frac12$, where $G = V_2^*V_1$ is a
  contraction.  By construction $G$ has the properties in the last
  statement, and clearly such an operator is unique.
\end{proof}

\begin{lemma}\label{sep26c}
  Let $\eH$ be a Hilbert space, and suppose $T\in\lh$, $T \ge 0$.  Let
  $\eK$ be a closed subspace of~$\eH$, and write
  \begin{equation*}
    T = 
     \begin{pmatrix}
       A & B^* \\ B & C
     \end{pmatrix} \colon \eK \oplus \eK^\perp \to \eK \oplus \eK^\perp .
  \end{equation*}
  Then there is a largest operator $S \ge 0$ in $\lk$ such that
  \begin{equation}\label{sep27c}
    \begin{pmatrix}
    A -S & B^* \\ B & C
  \end{pmatrix}  \ge 0.
  \end{equation}
  It is given by $S = A^\frac12 (I - G^*G) A^\frac12$, where $G \in
  \eL(\eK,\eK^\perp)$ is a contraction which maps $\ranclosure A$ into
  $\ranclosure C$ and is zero on the orthogonal complement of
  $\ranclosure A$.
\end{lemma}

\begin{proof}
  By Lemma \ref{sep27b}, we may define $S= A^\frac12 (I - G^*G)
  A^\frac12$ with $G$ as in the last statement of the lemma.
  Then
  \begin{equation*}
      \begin{pmatrix}
      A-S & B^* \\ B & C
    \end{pmatrix}
=
  \begin{pmatrix}
  A^\frac12 G^*G A^\frac12 & A^\frac12 G^* C^\frac12 \\
  C^\frac12 GA^\frac12 & C
  \end{pmatrix}
=
  \begin{pmatrix}
    A^\frac12 G^* \\ C^\frac12
  \end{pmatrix}
  \begin{pmatrix}
    GA^\frac12 & C^\frac12
  \end{pmatrix}
\ge 0.
  \end{equation*}
  Consider any $X \ge 0$ in $\lk$ such that 
  \begin{equation*}
    \begin{pmatrix}
      A-X & B^* \\ B & C
    \end{pmatrix} \ge 0.
  \end{equation*}
  Since $A \ge X \ge 0$, we can write $X = A^\frac12 K A^\frac12$
  where $K\in\lk$ and $0 \le K \le I$.  We can choose $K$ so that it
  maps $\ranclosure A$ into itself and is zero on $(\ranclosure
  A)^\perp$.  Then
  \begin{align*}
    \begin{pmatrix}
      A-X & B^* \\ B & C
    \end{pmatrix}
&=
  \begin{pmatrix}
    A - A^\frac12 K A^\frac12 & A^\frac12 G^* C^\frac12 \\
    C^\frac12 G A^\frac12 & C
  \end{pmatrix} \\
&=
  \begin{pmatrix}
    A^\frac12 & 0 \\ 0 & C^\frac12
  \end{pmatrix}
  \begin{pmatrix}
    I-K &G^* \\ G & I
  \end{pmatrix}
  \begin{pmatrix}
    A^\frac12 & 0 \\ 0 & C^\frac12    
  \end{pmatrix} .
  \end{align*}
  By our choices $G$ and $K$, we deduce that
  \begin{equation*}
    \begin{pmatrix} I-K &G^* \\ G & I \end{pmatrix} \ge 0.
  \end{equation*}
  By Lemma \ref{sep27b}, $G = G_1 (I-K)^\frac12$ where $G_1 \in
  \eL(\eK,\eK^\perp)$ is a contraction.  Therefore $G^*G \le I-K$, and
  so
  \begin{equation*}
    X = A^\frac12 K A^\frac12 \le A^\frac12 (I-G^*G) A^\frac12 = S.
  \end{equation*}
  This shows $S$ is maximal with respect to the property
  \eqref{sep27c}.
\end{proof}

\bibliographystyle{amsplain}
\bibliography{fr}

\end{document}